\documentclass{article}
\usepackage[english]{babel}

\usepackage[letterpaper,top=2cm,bottom=2cm,left=2cm,right=2cm,marginparwidth=1.75cm]{geometry}

\usepackage{calrsfs}
\usepackage{amssymb}
\usepackage{amsmath}
\usepackage{amsthm}
\usepackage{subcaption}
\usepackage{multirow}
\usepackage{graphicx}
\usepackage{amsfonts}
\usepackage{mathtools}
\usepackage{mathrsfs}
\usepackage{placeins}

\newtheorem{remark}{Remark}
\newtheorem{definition}{Definition}
\newtheorem{assumption}{Assumption}
\newtheorem{theorem}{Theorem}
\newtheorem{proposition}{Proposition}

\usepackage[mathscr]{euscript}
\usepackage[dvipsnames]{xcolor}
\usepackage[colorlinks=true, allcolors=OliveGreen]{hyperref}
\usepackage{filecontents}
\usepackage{tikz}
\usepackage{pgfplots}
\usetikzlibrary{external}
\usepackage{authblk}
\usepackage{todonotes}
\usepackage{array}
\newcolumntype{C}{>{\centering\arraybackslash}p{2.5cm}}

\newcommand{\averagel}{\{\!\!\{}
\newcommand{\averager}{\}\!\!\}}

\newcommand{\jumpl}{[\![}
\newcommand{\jumpr}{]\!]}

\newcommand{\partition}{\mathcal{T}_h}
\newcommand{\facesinternal}{\mathcal{F}^\mathrm{I}_h}
\newcommand{\faces}{\mathcal{F}_h}
\newcommand{\facesN}{\mathcal{F}_h^N}

\newcommand{\facesD}{\mathcal{F}_h^D}

\newcommand{\facesboundary}{\mathcal{F}^\mathrm{B}_h}

\newcommand{\Wh}{W_h^\mathrm{DG}}
\DeclareMathAlphabet{\mathcalligra}{T1}{calligra}{m}{n}

\title{Discontinuous Galerkin Methods for Fisher-Kolmogorov Equation with Application to $\alpha$-Synuclein Spreading in Parkinson's Disease \footnote{\textbf{Funding}: PFA has been partially funded by PRIN2017 research grant n. 201744KLJL funded by MUR. PFA, LD and AMQ have been partially funded by PRIN2020 research grant n. 20204LN5N5 funded by MUR. PFA has been partially supported by ICSC—Centro Nazionale di Ricerca in High Performance Computing, Big Data, and Quantum Computing funded by European Union—NextGenerationEU. The present research has been supported by MUR, grant Dipartimento di Eccellenza 2023-2027. MC, PFA, FB, LD and AMQ are members of INdAM-GNCS.}}

\author[1]{Mattia Corti}

\affil[1]{MOX-Dipartimento di Matematica, Politecnico di Milano, Piazza Leonardo da Vinci 32, Milan, 20133, Italy}
            
\author[1]{Francesca Bonizzoni}
\author[1]{Luca Dede'}
\author[1,2]{Alfio M. Quarteroni}
\author[1]{Paola F. Antonietti}

\affil[2]{Institute of Mathematics, \'{E}cole Polytechnique F\'{e}d\'{e}rale de Lausanne, Station 8, Av. Piccard, Lausanne, CH-1015, Switzerland (Professor Emeritus)}

\begin{document}
\maketitle

\begin{abstract}
This spreading of prion proteins is at the basis of brain neurodegeneration. This paper deals with the numerical modelling of the misfolding process of $\alpha$-synuclein in Parkinson’s disease. We introduce and analyze a discontinuous Galerkin method for the semi-discrete approximation of the Fisher-Kolmogorov (FK) equation that can be employed to model the process. We employ a discontinuous Galerkin method on polygonal and polyhedral grids (PolyDG) for space discretization, to accurately simulate the wavefronts typically observed in the prionic spreading and we prove stability and a priori error estimates. Next, we use a Crank-Nicolson scheme to advance in time. For the numerical verification of our numerical model, we first consider a manufactured solution, and then we consider a case with wavefront propagation in two-dimensional polygonal grids. Next, we carry out a simulation of $\alpha$-synuclein spreading in a two-dimensional brain slice in the sagittal plane with a polygonal agglomerated grid that takes full advantage of the flexibility of PolyDG approximation. Finally, we present a simulation in a three-dimensional geometry reconstructed from magnetic resonance images of a patient's brain.
\end{abstract}

\section{Introduction}
Neurodegeneration represents a major challenge because of the ageing trends in the worldwide population. Evidence suggests that the misfolding and aggregation of prionic proteins into toxic and insoluble conformations stand at the basis of neurodegeneration \cite{walker_neurodegenerative_2015}. A most common protein undergoing the misfolding process is the $\alpha$-synuclein protein \cite{alafuzoff_chapter_2018}. In the literature, this protein is known to be related to many different pathologies, known as $\alpha$-synuclopathies, such as Parkinson's disease \cite{stefanis-alpha-syn}, Parkinson's disease with dementia and dementia with Lewy bodies \cite{breitve_longitudinal_2018}.
\par
To better highlight the differences between these pathologies (often co-existing), in recent years several mathematical models for the prion dynamics have been proposed. A mathematical description of the spreading of prionic proteins is of primary importance, particularly for $\alpha$-synuclein, for which positron emission tomography imaging (PET) cannot be used in clinical practice due to the absence of chemical ligands \cite{korat_alpha-synuclein_2021}. In \cite{schiesser:parkinson}, the author studied a coupling of ordinary differential equations (ODE) models for the microscopical dynamic inside the neuron and partial differential equations (PDE) models for the macroscopic spreading. Due to the oligomer coagulation and fragmentation phenomena, many works are based on the Smoluchowky coagulation equation \cite{bertsch_alzheimers_2017, franchi-lorenzani}. However, their overarching numerical inherited complexity calls for suitable simplifications eventually leading to simpler diffusion-reaction problems to be solved on the whole brain geometry; an example is provided by the Fisher-Kolmogorov (FK) model (also known as Fisher-KPP model) \cite{weickenmeierPhysicsbasedModelExplains2019, fornari_prion-like_2019}. The latter \cite{fisher-1937, kolmogorov-1937} is a nonlinear diffusion-reaction equation applied in many different contexts, in particular biological species' evolution.
\par
Over the years, many different numerical methods, such as Finite Element Methods (FEM) \cite{weickenmeierPhysicsbasedModelExplains2019, roessler_numerical_1997, engwer_estimating_2021}, Finite Differences Methods (FDM) \cite{macias-diaz_explicit_2012}, and Boundary Elements Methods \cite{gortsas_local_2022} were proposed to compute the approximate solution of the FK equation. A structure-preserving Discontinuous Galerkin (DG) formulation was proposed in \cite{bonizzoni_structure-preserving_2020}, where a change of variable is developed to preserve the positivity of the numerical solution.
\par
To face the geometric complexity and the need for high-order accuracy, in this work, we propose and analyze a Discontinuous Galerkin formulation on polygonal/polyhedral grids (PolyDG) for the semi-discrete approximation of the FK equation coupled with an implicit second-order in time discretization. The typical solution of the equation is a propagating wavefront: to capture it, a mesh with sufficient refinement is needed. The simplicity in supporting high-order approximations as well as the flexibility in handling complex geometries, and locally varying discretization parameters justify the choice of a PolyDG approach \cite{antonietti_highorder_2021}.
\par
Another strength of the proposed formulation is its flexibility in mesh generation; due to its applicability to polygonal/polyhedral meshes. The geometrical complexity of the brain is a challenge in mesh construction. The possibility of refining the mesh only in some regions, handling the hanging nodes and eventually using arbitrarily shaped elements is easy to implement in our approach. A powerful tool is also offered by mesh agglomeration \cite{manuzzi:CNN, corti:MPET}. In this setting, starting from an initial very detailed mesh, it is possible to generate a coarse one composed of generic polygons, which preserves the original detail of the boundary representation without the need for curved elements and with a reduction of the computational cost. Our formulation allows the accurate approximation of the wavefront velocity generated by the FK equation on this type of mesh velocity, which is a very desirable property in this context.
\par
\bigskip
The paper is organized as follows. In Section \ref{sec:model}, we introduce the FK mathematical model and discuss its application to neurodegeneration. In Section \ref{sec:polydg}, we introduce the PolyDG space discretization of the problem. In Section \ref{sec:stability}, we prove the stability of the semi-discretized problem. Section \ref{sec:error} is devoted to the proof of a priori error estimates of the semi-discretized problem. In Section \ref{sec:temporaldisc}, we introduce the time-discretization by means of the Crank-Nicolson method. In Section \ref{sec:numericalresults}, we validate the theoretical results by presenting some convergence tests to our manufactured solutions. Moreover, we assess the accuracy of the proposed scheme in capturing travelling wave on a two-dimensional setting. Finally, we analyse the application of $\alpha$-synuclein spreading in Parkinson's disease both in two-dimensional (with agglomerated polygonal meshes) and three-dimensional real geometries reconstructed from medical images. Finally, in Section  \ref{sec:conclusion}, we draw some conclusions and discuss further developments.

\section{The mathematical model}
\label{sec:model}
In this section, we consider the FK equation to describe the reaction and diffusion of misfolded proteins. For a final time $T>0$, the problem is dependent on time $t\in(0,T]$ and space $\boldsymbol{x}\in\Omega\subset\mathbb{R}^d$ ($d=2,3$). The solution to our problem represents the relative concentration of the misfolded protein $c=c(\boldsymbol{x},t)$. Indeed, under the assumption of constant baseline concentration of healthy state protein, the variable $c$ is rescaled in the interval $[0,1]$, where $0$ means the absence of misfolded proteins and $1$ is the high prevalence of them. A detailed derivation of this model can be found in \cite{weickenmeierPhysicsbasedModelExplains2019}.
\par
The problem in its strong formulation reads as follows:
\begin{equation}
 \begin{dcases}
     \dfrac{\partial c}{\partial t} =\nabla \cdot(\mathbf{D} \nabla\, c) + \alpha\,c(1-c) + f,
    & \mathrm{in}\,\Omega\times(0,T],
    \\[8pt]
    (\mathbf{D}\nabla c) \cdot \boldsymbol{n} = \phi_\mathrm{N}, & 
    \mathrm{on}\;\Gamma_N\times(0,T],
    \\[8pt]
    c = c_\mathrm{D}, & \mathrm{on}\;\Gamma_D\times(0,T],
    \\[8pt]
    c(\boldsymbol{x},0)=c_0(\boldsymbol{x}), & \mathrm{in}\;\Omega.
    \\[8pt]
\end{dcases}
\label{eq:fk_strong}
\end{equation}

In Equation \eqref{eq:fk_strong}, the reaction parameter $\alpha=\alpha(\boldsymbol{x})$ represents the local conversion rate of the proteins from healthy to misfolded state. Moreover, the diffusion tensor $\mathbf{D}=\mathbf{D}(\boldsymbol{x})$ denotes the spreading of misfolded protein inside the domain (the whole brain parenchymal tissue in our case), and the function $f=f(\boldsymbol{x},t)$ is a forcing term, which models the external addition/removal of mass (e.g. modelling some clearance mechanisms). Concerning the boundary conditions, we impose a sufficiently regular flux $\phi_\mathrm{N}$ on the boundary $\Gamma_N$ of the domain, while $c_\mathrm{D}$ fixes a value of concentration on a part of the boundary $\Gamma_D$. We underline that $\Gamma_D \cup \Gamma_N = \partial \Omega$ and $\Gamma_D \cap \Gamma_N = \emptyset$. 
\par
In the prions' spreading applications, the diffusion tensor is typically modelled as the superimposition of an extracellular diffusion effect with magnitude $d_\mathrm{ext}$ and an axonal diffusion with magnitude $d_\mathrm{axn}$ \cite{weickenmeierPhysicsbasedModelExplains2019}; for this reason, we assume that $\mathbf{D}$ has the following structure:
\begin{equation}
\label{eq:difftensor}
    \mathbf{D} = d_\mathrm{ext}\mathbf{I} + d_\mathrm{axn}(\boldsymbol{n}\otimes \boldsymbol{n}),
\end{equation}
where $\boldsymbol{n}=\boldsymbol{n}(\boldsymbol{x})$ is the axonal fibres direction at the point $\boldsymbol{x}\in\Omega$ and $d_\mathrm{ext}, d_\mathrm{axn} \geq 0$. The axonal direction is the principal orientation of the connections between the neurons (axons), which can be derived from Diffusion Tensor Imaging (DTI). The derivation of these directions is of primary importance for our purposes, because most of the spreading of the prions occurs through the axons \cite{weickenmeierPhysicsbasedModelExplains2019}. 
\par
We assume that the subset $\Gamma_D\subset\partial\Omega$ introduced above has positive measure $|\Gamma_D|>0$, then we define the Sobolev spaces $ W_0:=H^1_{\Gamma_D}(\Omega) = \{w\in H^1(\Omega):\; w|_{\Gamma_D}=0\}$ and $ W_\mathrm{D}:= \{w\in H^1(\Omega):\; w|_{\Gamma_D}=c_\mathrm{D}\}$. When $|\Gamma_D|=0$ and $\Gamma_N=\partial \Omega$, we define $W_0=W_\mathrm{D}=H^1(\Omega)$. Moreover, we employ a standard definition of scalar product in $L^2(\Omega)$, denoted by $(\cdot,\cdot)_\Omega$. The induced norm is denoted by $||\cdot||_\Omega$. For vector-valued and tensor-valued functions the definition extends componentwise \cite{salsa:EDP}. Given $k\in\mathbb{N}$ and an Hilbert space $H$ we use the notation $C^k([0,T],H)$ to denote the space of functions $c = c(\boldsymbol{x},t)$ such that $c$ is $k$-times continuously differentiable with respect to time and for each $t\in[0,T]$, $c(\cdot,t)\in H$, see e.g. \cite{salsa:EDP}. Adopting standard notation for Sobolev spaces, we make the following assumption on the coefficients' regularity.
\begin{assumption}[Coefficients' regularity]
We assume the following regularities for the coefficients and the forcing terms appearing in \eqref{eq:fk_strong}: 
\begin{itemize}
    \item $\alpha\in L^\infty(\Omega)$. 
    \item $\boldsymbol{\mathrm{D}}\in L^\infty(\Omega,\mathbb{R}^{d\times d})$ and $\exists d_0>0\;\forall\boldsymbol{\xi}\in \mathbb{R}^d:\; d_0|\boldsymbol{\xi}|^2 \leq \boldsymbol{\xi}^\top\mathbf{D}\boldsymbol{\xi} \quad \forall \boldsymbol{\xi}\in\mathbb{R}^d$.
    \item $f\in L^2((0,T],L^2(\Omega))$.
    \item $\phi_\mathrm{N}\in L^2((0,T];L^2(\Gamma_N))$.
    \item $c_\mathrm{D} \in L^2((0,T]; H^{1/2}(\Gamma_D))$.
    \item $c_0 \in L^2(\Omega)$.
\end{itemize} 
\end{assumption}
It can be proved that, under Assumption 1 and if $f=0$, $\phi_\mathrm{N}=0$, and $\Gamma_D=\emptyset$, when we consider $c_0(\boldsymbol{x})\in[0,1]$ for each $\boldsymbol{x}\in\Omega$, the FK equation admits a travelling wave solution. Moreover, it can be also proved that under these assumptions: $c(\boldsymbol{x},t)\in[0,1]$ for each $\boldsymbol{x}\in\Omega$ and $t>0$. In this specific setting, the equations admit two steady-state solutions: an unstable equilibrium at $c = 0$ and a stable one at $c = 1$. This implies:
\begin{equation*}
    \lim_{t\rightarrow +\infty} c(\boldsymbol{x},t)=1 \qquad \text{if}\;\exists \boldsymbol{x}\in\Omega \; \mathrm{such\;that}\; c_0(\boldsymbol{x})>0.
\end{equation*}

\par
By setting:
    \begin{equation}
        a(c,w) =  \left(\sqrt{\mathbf{D}} \nabla c,\sqrt{\mathbf{D}}\nabla w\right)_\Omega \qquad \forall c, w\in W
    \end{equation}
    \begin{equation}
        r_L(c,w) = \left(\alpha c, w\right)_\Omega \quad \forall c,w\in W,
    \end{equation}
    \begin{equation}
        r_N(v,c,w) = \left(\alpha(v c), w\right)_\Omega \quad \forall c,w,v\in W,
    \end{equation}                          
    \begin{equation}
        F(w) = (f,w)_\Omega + \left(\phi_\mathrm{N},w\right)_{\Gamma_N} \qquad \forall w\in W, %
    \end{equation}
the weak formulation of problem \eqref{eq:fk_strong} reads:
\par
\bigskip
For each $t\in(0,T]$ find $c(\boldsymbol{x},t)\in W_\mathrm{D}$ such that:
\begin{equation}
\begin{dcases}
     \left(\dfrac{\partial c(\boldsymbol{x},t)}{\partial t},w\right)_\Omega + a(c(\boldsymbol{x},t),w) - r_L(c(\boldsymbol{x},t),w) + r_N(c(\boldsymbol{x},t),c(\boldsymbol{x},t),w) = F(w) 
     & \forall w\in W_0, \\[8pt]
    c(\boldsymbol{x},0)=c_0, & \mathrm{in}\;\Omega.
\end{dcases}
\label{eq:weakform}
\end{equation}

\section{PolyDG semi-discrete formulation}
\label{sec:polydg}
In this section, after defining some preliminary concepts, we approximate in space the FK equation by the PolyDG method. For the sake of simplicity, we neglect the dependencies of the inequality constants on the model parameters, using the notation $x\lesssim y$ to mean that $\exists C>0: x\leq C y$, where $C$ depends on the model parameters (it may depend on $p$, but it is independent of the discretization parameter $h$).

\subsection{Discrete setting and preliminary estimates}
Let us introduce a polytopic mesh partition $\partition$ of the domain $\Omega$ made of disjoint polygonal/polyhedral elements $K$, where for each element $K\in \partition$, we denote by $|K|$ the measure of the element and by $h_K$ its diameter. We set $h=\max_{K\in\partition} h_K<1$. We define the interface as the intersection of the $(d-1)-$dimensional facets of two neighbouring elements. We distinguish two cases:
\begin{itemize}
    \item case $d=2$, in which the interfaces are always line segments; then we denote such a set of segments with $\faces$.
    \item case $d=3$, in which any interface consists of a generic polygon, we further assume that we can decompose each interface into (planar) triangles; we denote the set of all these triangles with $\faces$;
\end{itemize}
It is now useful to decompose $\faces$ into the union of interior faces ($\facesinternal$) and exterior faces ($\facesboundary$ ) lying on the boundary of the domain $\partial\Omega$, i.e. $\faces = \facesinternal \cup \facesboundary$. Moreover, the boundary faces set can be split according to the type of imposed boundary condition: $\facesboundary = \facesD \cup \facesN$, where $\facesD$ and $\facesN$ are the boundary faces contained in $\Gamma_D$ and $\Gamma_N$, respectively.  We assume that $\partition$ is aligned with $\Gamma_D$ and $\Gamma_N$, i.e. any $F\in\facesboundary$ is contained in either $\Gamma_D$ or $\Gamma_N$.
\par
\begin{assumption}
The mesh sequence $\{\partition\}_h$ satisfies the following properties \cite{dipietro:HHO}:
\begin{enumerate}
    \item Shape Regularity: $\forall K\in\partition\;it\;holds: c_1 h_K^d\lesssim q|K|\lesssim  c_2h_K^d$.
    \item Contact Regularity: $\forall F\in\faces$ with $F\subseteq \overline{K}$ for some $K\in\partition$, it holds $h_K^{d-1}\lesssim |F|$, where $|F|$ is the Hausdorff measure of the face $F$.
    \item Submesh Condition: There exists a shape-regular, conforming, matching simplicial submesh $\widetilde{\partition}$ such that:
    \begin{itemize}
        \item $\forall \widetilde{K}\in\widetilde{\partition}\;\exists K\in\partition:\quad \widetilde{K}\subseteq K$.
        \item The family $\{\widetilde{\partition}\}_h$ is shape and contact regular.
        \item $\forall \widetilde{K}\in\widetilde{\partition}, K\in\partition$ with $\widetilde{K} \subseteq K$, it holds $h_K \lesssim h_{\widetilde{K}}$.
    \end{itemize}
\end{enumerate}
\end{assumption}
Let us define $\mathbb{P}_{p}(K)$ as the space of polynomials of total degree $p\geq 1$ over a mesh element $K$. Then we can introduce the following discontinuous finite element space:
\begin{equation*}
    \Wh = \{w\in L^2(\Omega):\quad w|_K\in\mathbb{P}_{p}(K)\quad\forall K\in\partition\}
\end{equation*}
\par
\par
We next introduce the so-called trace operators \cite{arnoldUnifiedAnalysisDiscontinuous2001}. Let $F\in\facesinternal$ be a face shared by the elements $K^\pm$. Let $\boldsymbol{n}^\pm$ by the unit normal vector on face $F$ pointing exterior to $K^\pm$, respectively. Then, for sufficiently regular scalar-valued functions $v$ and vector-valued functions $\boldsymbol{q}$ respectively, we define:
\begin{itemize}
    \item the average operator $\averagel{\cdot}\averager$ on $F\in \facesinternal$: $\averagel{v}\averager = \dfrac{1}{2} (v^+ + v^-), \quad \averagel{\boldsymbol{q}}\averager = \dfrac{1}{2} (\boldsymbol{q}^+ + \boldsymbol{q}^-)$;
    \item the jump operator $\jumpl{\cdot}\jumpr$ on $F\in \facesinternal$: $\jumpl{v}\jumpr = v^+\boldsymbol{n}^+ + v^-\boldsymbol{n}^-, \quad \jumpl{\boldsymbol{q}}\jumpr = \boldsymbol{q}^+\cdot\boldsymbol{n}^+ + \boldsymbol{q}^-\cdot\boldsymbol{n}$.
\end{itemize}
In these relations we are using the superscripts $\pm$ on the functions, to denote the traces of the functions on $F$ taken within the interior to $K^\pm$. We remark that the jump of a scalar is a vector and the jump of a vector is a scalar. In the same way, we can define analogous operators on the face $F\in\facesD$ associated with the cell $K\in\partition$ with $\boldsymbol{n}$ outward unit normal on $\partial\Omega$:
\begin{itemize}
    \item the average operator $\averagel{\cdot}\averager$ on $F\in\facesD$: $\averagel{v}\averager = v, \quad \averagel{\boldsymbol{q}}\averager = \boldsymbol{q}$;
    \item the standard jump operator $\jumpl{\cdot}\jumpr$ on $F\in\facesD$, with Dirichlet conditions $g$, $\boldsymbol{g}$: $\jumpl{v}\jumpr = (v-g)\boldsymbol{n}, \quad \jumpl{\boldsymbol{q}}\jumpr = (\boldsymbol{q}-\boldsymbol{g})\cdot\boldsymbol{n}$.
\end{itemize}
We recall the following identity that will be useful in the method derivation:
\begin{equation}
    \jumpl v \boldsymbol{q} \jumpr = \jumpl \boldsymbol{q} \jumpr \averagel v \averager + \averagel \boldsymbol{q} \averager \cdot \jumpl v \jumpr \qquad \forall F \in \facesinternal.
\end{equation}
\par
Let us introduce the following broken Sobolev spaces for an integer $r\geq1$: $H^r(\mathcal{T}_h) = \{w_h\in L^2(\Omega): w_h|_K\in H^r(K)\quad \forall K\in\mathcal{T}_h\}$. Moreover, we introduce the shorthand notation for the $L^2$-norm $||\cdot||:=||\cdot||_{L^2(\Omega)}$ and for the $L^2$-norm on a set of faces $\mathcal{F}$ as $||\cdot||_\mathcal{F}=\left(\sum_{F\in\mathcal{F}}||\cdot||_{L^2(F)}^2\right)^{1/2}$. We define the following DG-norm:
\begin{equation}
    ||c||_{\mathrm{DG}} = \Big|\Big|\sqrt{\mathbf{D}}\nabla_h c \Big|\Big| + ||\sqrt{\eta}\jumpl c\jumpr||_{\facesinternal\cup\facesD} \qquad \forall c\in H^1(\partition)
\end{equation}
Furthermore, we recall the discrete Gagliardo-Nirenberg inequality \cite{nirenbergdiscrete}:
\begin{equation}
    \forall u_h\in \Wh \quad \exists C_{\mathrm{G}_d} = C_{\mathrm{G}_d}(p) >0:\qquad||u_h||_{L^q(\Omega)}\leq C_{\mathrm{G}_d} ||u_h||_{\mathrm{DG}}^s\;||u_h||_{L^2(\Omega)}^{1-s},
    \label{eq:gagliardo}
\end{equation}
with $s \in [0,1]$ and $q$ such that:
\begin{equation}
    \dfrac{1}{q} = s\left(\dfrac{1}{2} - \dfrac{1}{d}\right) + \dfrac{1-s}{2}.
\end{equation}
\begin{remark}
    We remark that most of the analysis is valid also for milder assumptions on the mesh than the ones in Assumption 2, which could be only polytopic regular \cite{Cangiani:PolyDG}. However, Gagliardo's inequality would not be valid anymore, hence we cannot completely extend the analysis.
\end{remark}
In this equation, the constant $C_{\mathrm{G}_d}$ is independent of the discretization parameter $h$. Finally, we recall the Perov inequality \cite{Perov:inequality}, that we use as an extension of the Gr\"{o}nwall inequality.
\begin{proposition}[Perov Inequality]
    Let $a,b,c$ be three positive constants and let $u\in L^\infty_+(0,\hat{t}):=\{ u\in L^\infty(0,\hat{t}): \; u(t)\geq0$ a.e. in $(0,\hat{t})\}$ such that:
    \begin{equation}
        u(t) \leq a + b\int_0^t u(s) \mathrm{d}s + c \int_0^t u^\gamma(s) \mathrm{d}s, \qquad \mathrm{for\;almost\;any}\; t\in(0,\hat{t}),
    \end{equation}
where $\hat{t}$ is such that:
\begin{equation}
    e^{b(\gamma-1)\hat{t}} < 1 + \dfrac{b}{a^{\gamma-1}c}.
    \label{eq:timebound}
\end{equation}
Then for almost any $t\in(0,\hat{t})$ we have:
\begin{equation}
    u(t) \leq \dfrac{a\;e^{bt}}{\left(1-a^{\gamma-1}\,c\,b^{-1}\left(e^{b(\gamma-1)t}-1\right)\right)^\frac{1}{\gamma-1}}
\end{equation}
\end{proposition}
\begin{remark}
    If $c=0$ we recover the classical Gr\"{o}nwall inequality estimate: $u(t) \leq a\,e^{bt},\; \forall t\leq T$.
\end{remark}
\subsection{PolyDG semi-discrete formulation}
To construct the semi-discrete formulation, we define the following penalization function $\eta:\faces\rightarrow\mathbb{R}_+$:
\begin{equation}
    \eta = \eta_0
    \begin{cases}
         \dfrac{p^2}{\{h\}_\mathrm{H}},  & \mathrm{on}\; F\in\facesinternal\\
         \dfrac{p^2}{h},                 & \mathrm{on}\; F\in\facesD
    \end{cases}.
    \label{eq:penalty}
\end{equation}
In Equation \eqref{eq:penalty}, we are considering the harmonic average operator $\{\cdot\}_\mathrm{H}$ on $F\in\facesinternal$ and $\eta_0$ is a parameter at our disposal (to be chosen large enough to have stability). Moreover, we define the bilinear form $\mathcal{A}:\Wh\times \Wh\rightarrow \mathbb{R}$ as:
\begin{equation}
    \mathcal{A}(c,w) = \int_{\Omega}\nabla_h c\cdot\nabla_h w +\sum_{F\in\facesinternal\cup\facesD}\int_{F}\left(\eta \jumpl c\jumpr \cdot \jumpl w\jumpr -\averagel\mathbf{D} \nabla c\averager \cdot \jumpl w \jumpr -  \jumpl c\jumpr \cdot \averagel\mathbf{D} \nabla w\averager\right)\mathrm{d}\sigma\qquad \forall c,w\in\Wh,
\end{equation}
where $\nabla_h \cdot$ is the elementwise gradient \cite{quarteroni:EDP}. The semi-discrete PolyDG formulation reads.
\par
\bigskip
Find $c_h(t)\in \Wh$ such that $\forall t>0$:
\begin{equation}
\begin{dcases}
     \left(\dfrac{\partial c_h(t)}{\partial t},w_h\right)_\Omega + \mathcal{A}(c_h(t),w_h) - r_L(c_h(t),w_h) + r_N(c_h(t),c_h(t),w_h) = F(w_h)
     & \forall w_h\in \Wh, \\[8pt]
    c_h(0)=c_{h}^0 & \mathrm{in}\;\Omega_h,
\end{dcases}
\label{eq:DGFormulation}
\end{equation}
where $c_{h}^0\in\Wh$ is a suitable approximation of $c_0$. Its derivation follows the classical steps of the DG formulation for the Laplace equation (see \cite{quarteroni:EDP}). For more details on the definition of numerical fluxes associated with the symmetric interior penalty DG method considered in this paper, we refer to \cite{arnoldUnifiedAnalysisDiscontinuous2001}.

\section{Stability analysis of the semi-discrete formulation}
\label{sec:stability}
For the analysis, we exploit continuity and coercivity of the bilinear form $\mathcal{A}(\cdot,\cdot)$. The proof of these properties can be found in \cite{houston:book}. Concerning the well-posedness of the formulation we refer to analysis of semilinear parabolic formulations in FEM \cite{FKErrors} and DG \cite{suli:parabolic} settings.
\par
For simplicity, in both stability and convergence analyses, we assume homogeneous Dirichlet ($c_\mathrm{D}=0$) and Neumann ($\phi_\mathrm{N}=0$) boundary conditions.
\begin{proposition}
Let Assumption 2 be satisfied, then the bilinear form $\mathcal{A}(\cdot,\cdot)$ is continuous and coercive:
\begin{equation}
    \exists M > 0: \quad |\mathcal{A}(v_h,w_h)| \leq M ||v_h||_\mathrm{DG} ||w_h||_\mathrm{DG} \qquad \forall v_h,w_h \in \Wh,
    \label{eq:continuity}
\end{equation}
\begin{equation}
    \exists \mu > 0: \quad \mathcal{A}(v_h,v_h) \geq \mu ||v_h||_\mathrm{DG}^2 \qquad \forall v_h\in \Wh,
    \label{eq:coercivity}
\end{equation}
where $M$ and $\mu$ are independent of $h$. Coercivity holds provided that the penalty parameter $\eta$ is large enough.
\end{proposition}
\begin{definition}[Energy norm]
The energy norm $||\cdot||_{\epsilon}:H^1(\partition)\rightarrow\mathbb{R}$ is defined as:
\begin{equation}
\label{eq:energynorm}
    ||c_h(t)||_{\epsilon}^2 := ||c_h(t)||^2  + \int_0^t ||c_h(s)||_\mathrm{DG}^2 \mathrm{d}s
\end{equation}
\end{definition}
\begin{theorem}[Stability estimate]
Let Assumptions 1 and 2 be satisfied and, for a sufficiently large penalty parameter $\eta$, let $c_h(t)$ be the solution of Equation \eqref{eq:DGFormulation} for any $t\in(0,\hat{t}]$, with $\hat{t}\leq T$ introduced in \eqref{eq:timebound}. Then:
\begin{equation}
        ||c_h(t)||_{\epsilon}^2 \; \leq \dfrac{\left(||c_h^0||^2 + \displaystyle\int_0^T ||f(s)||^2   \mathrm{d}s\right) e^{\frac{2\tilde{\alpha}+1}{\tilde{\mu}}t}}{\left(\tilde{\mu}^{d-1}-\dfrac{\tilde{\alpha}C_{\mathrm{G}_d}^3}{2^{d-1}(2\tilde{\alpha}+1)\varepsilon}\left(||c_h^0||^2 + \displaystyle\int_0^T ||f(s)||^2 \mathrm{d}s\right)^{d-1}\left(e^{\left(\frac{2\tilde{\alpha}+1}{\tilde{\mu}}\right)(d-1)t}-1\right)\right)^\frac{1}{d-1}},
\end{equation}
where $\tilde{\mu} = \min\left\{1,2\mu- \frac{d\varepsilon \tilde{\alpha}C_{\mathrm{G}_d}^3}{(2^{d-2})}\right\}>0$ and $\varepsilon>0$ is small enough, $\tilde{\alpha}=||\alpha||_{L^\infty}$ and $C_{\mathrm{G}_d}$ defined in Equation \eqref{eq:gagliardo}.
\end{theorem}
\begin{proof}
We start from the Equation \eqref{eq:DGFormulation} and we choose $w_h=c_h(t)$, to find:
\begin{equation*}
    \left(\dot{c}_h,c_h\right)_\Omega + \mathcal{A}(c_h,c_h) - r_L(c_h,c_h) + r_N(c_h,c_h,c_h) = F(c_h),
\end{equation*}
where we are using the notation of time derivative $\dot{c}_h=\partial{c}_h/\partial{t}$.
Then after integration in time of the equation above, we can use the coercivity estimate in \eqref{eq:coercivity}, and the H\"{o}lder inequality with the definition of $\tilde{\alpha} = ||\alpha||_{L^\infty}$, to obtain:
\begin{equation*}
    ||c_h(t)||^2 -  ||c_h^0||^2 + \int_0^t 2\mu ||c_h(s)||_\mathrm{DG}^2 \mathrm{d}s \leq  \int_0^t 2\tilde{\alpha}||c_h(s)||^2 \mathrm{d}s + \int_0^t 2\tilde{\alpha} ||c_h(s)||^3_{L^3(\Omega)} \mathrm{d}s + \int_0^t 2 ||f(s)|| \,||c_h(s)||  \mathrm{d}s.
\end{equation*}
Then we need to use the Gagliardo-Nirenberg inequality \eqref{eq:gagliardo}, on each element $K\in\partition$. For this reason, we need to distinguish between the cases $d=2$ and $d=3$.
\begin{description}
    \item [Case $d=2$:] In this case the inequality \eqref{eq:gagliardo} applies with $s= 1/3$. By applying Young's inequality, we obtain:
    \begin{equation*}
        ||c_h||_{L^3(\Omega)}^3 \leq \left(C_{\mathrm{G}_2}||c_h||^{\frac{2}{3}}||c_h||_\mathrm{DG}^{\frac{1}{3}}\right)^3 = C_{\mathrm{G}_2}^3||c_h||^2\;||c_h||_\mathrm{DG} \leq \dfrac{C_{\mathrm{G}_2}^3}{2}\left(\dfrac{1}{\varepsilon}||c_h||^4+\varepsilon||c_h||_\mathrm{DG}^2\right).
    \end{equation*}
    Using Assumption 1, we obtain:
\begin{equation*}
    ||c_h(t)||^2  + \int_0^t \left(2\mu-\dfrac{2\varepsilon \tilde{\alpha}C_{\mathrm{G}_2}^3}{2}\right)||c_h(s)||_\mathrm{DG}^2 \mathrm{d}s \; \leq ||c_h^0||^2 + \int_0^t (2\tilde{\alpha}||c_h(s)||^2 +  \dfrac{2\tilde{\alpha}C_{\mathrm{G}_2}^3}{2\varepsilon} ||c_h(s)||^4 + 2 ||f(s)||\,||c_h(s)||  \mathrm{d}s.
\end{equation*}

\item [Case $d=3$:] In this case inequality \eqref{eq:gagliardo} applies  with $s = 1/2$. By applying Young's inequality with $\gamma=3/4$ and $\gamma^*=1/4$ we get:
    \begin{equation*}
        ||c_h||_{L^3(\Omega)}^3 \leq \left(C_{\mathrm{G}_3}||c_h||^{\frac{1}{2}}||c_h||_{\mathrm{DG}}^{\frac{1}{2}}\right)^3 = C_{\mathrm{G}_3}^3||c_h||^{\frac{3}{2}}\;||c_h||_{\mathrm{DG}}^{\frac{3}{2}} \leq \dfrac{C_{\mathrm{G}_3}^3}{4}\left(\dfrac{1}{\varepsilon}||c_h||^6+3\varepsilon||c_h||_{\mathrm{DG}}^2\right).
    \end{equation*}
    Using Assumption 1, we obtain:
\begin{equation*}
    ||c_h(t)||^2  + \int_0^t \left(2\mu-\dfrac{6\varepsilon \tilde{\alpha}C_{\mathrm{G}_3}^3}{4}\right)||c_h(s)||_\mathrm{DG}^2 \mathrm{d}s \; \leq ||c_h^0||^2 + \int_0^t (2\tilde{\alpha}||c_h(s)||^2 + \dfrac{2\tilde{\alpha}C_{\mathrm{G}_3}^3}{4 \varepsilon} ||c_h(s)||^6 + 2||f(s)||\,||c_h(s)||  \mathrm{d}s.
\end{equation*}
\end{description}
Summarizing, using the definition \eqref{eq:energynorm} of the energy norm, and introducing $\tilde{\mu} = \min\{1,2\mu- d\varepsilon \tilde{\alpha}C_{\mathrm{G}_3}^3/(2^{d-2})\}$ we obtain:
\begin{equation*}
    \tilde{\mu}||c_h(t)||_{\epsilon}^2 \; \leq ||c_h^0||^2 + \int_0^t (2\tilde{\alpha}||c_h(s)||^2 + \dfrac{\tilde{\alpha}C_\mathrm{Gd}^3}{2^{d-1}\varepsilon} ||c_h(s)||^{2d} + 2 ||f(s)|| \,||c_h(s)||  \mathrm{d}s.
\end{equation*}
Then we apply Young's inequality to the forcing terms, and we exploit the positivity of the integrated to get:
\begin{equation*}
    \tilde{\mu}||c_h(t)||_{\epsilon}^2 \; \lesssim ||c_h^0||^2 + \displaystyle\int_0^T ||f(s)||^2 \mathrm{d}s + \int_0^t \left(\left(2\tilde{\alpha}+1\right)||c_h(s)||^2 + \dfrac{\tilde{\alpha}C_{\mathrm{G}_d}^3}{2^{d-1}\varepsilon}||c_h(s)||^{2d}\right) \mathrm{d}s.
\end{equation*}
Finally, we can apply Perov's inequality to conclude the proof:
\begin{equation*}
        ||c_h(t)||_{\epsilon}^2 \; \leq \dfrac{\left(||c_h^0||^2 + \displaystyle\int_0^T ||f(s)||^2 \mathrm{d}s\right) e^{\frac{2\tilde{\alpha}+1}{\tilde{\mu}}t}}{\left(\tilde{\mu}^{d-1}-\dfrac{\tilde{\alpha}C_{\mathrm{G}_d}^3}{2^{d-1}(2\tilde{\alpha}+1)\varepsilon}\left(||c_h^0||^2 + \displaystyle\int_0^T ||f(s)||^2 \right)^{d-1}\left(e^{\left(\frac{2\tilde{\alpha}+1}{\tilde{\mu}}\right)(d-1)t}-1\right)\right)^\frac{1}{d-1}}.
\end{equation*}
\end{proof}
\begin{remark}
   In view of the neurodegenerative modelling application, under the assumption $f=0$, the stability estimate of Theorem 1 reduces to:
    \begin{equation}
        ||c_h(t)||_{\epsilon}^2 \; \leq \dfrac{||c_h^0||^2 e^{\frac{2\tilde{\alpha}+1}{\tilde{\mu}}t}}{\left(\tilde{\mu}^{d-1}-\dfrac{\tilde{\alpha}C_{\mathrm{G}_d}^3||c_h^0||^{2d-2}}{2^{d-1}(2\tilde{\alpha}+1)\varepsilon}\left(e^{\left(\frac{2\tilde{\alpha}+1}{\tilde{\mu}}\right)(d-1)t}-1\right)\right)^\frac{1}{d-1}} =: C_S(c^0_{h}),
    \label{eq:simplifiedenergyestimate}    
    \end{equation}
    where $\tilde{\mu} = \min\left\{1,2\mu- \frac{d\varepsilon \tilde{\alpha}C_{\mathrm{G}_d}^3}{2^{d-2}}\right\}>0$ and $\varepsilon>0$ is small enough, $\tilde{\alpha}=||\alpha||_{L^\infty}$ and $C_{\mathrm{G}_d}$ defined in Equation \eqref{eq:gagliardo}. The definition of $C_S$ will be useful in the following analysis.
\end{remark}
\begin{remark}
     Since $\varepsilon$ can be chosen arbitrarily small in Young's inequality, the positivity of $\tilde{\mu}$ is always guaranteed and we do not have a structural relation between the parameters. At the same time, the positivity of the denominator is guaranteed for $t$ that satisfies relation \eqref{eq:timebound}, thanks to Perov inequality.
\end{remark}

\section{Error analysis of the semi-discrete formulation}
\label{sec:error}
In this section, we derive an a priori error estimate for the solution of the PolyDG semi-discrete problem \eqref{eq:DGFormulation}.
\par
First of all, we need to introduce the following definition:
\begin{equation}
    |||u|||_{\mathrm{DG}} = ||u||_{\mathrm{DG}}
    + \Big|\Big|\eta^{-\frac{1}{2}}\averagel \mathbf{D}\nabla_h u\averager\Big|\Big|_{\mathrm{L}^2(\facesinternal\cup\facesD)} \qquad\forall u \in H^2(\partition).
\end{equation}
We remark that it exists $\widetilde{C}>0$ such that $\widetilde{C}|||v|||_\mathrm{DG}^2\leq||v||_\mathrm{DG}^2$ for all $v\in\Wh$.
\par
We introduce the interpolant $c_\mathrm{I}\in\Wh$ of the continuous formulation \eqref{eq:weakform} \cite{babuska-interpolant}. 
\begin{proposition}
Let Assumption 2 be fulfilled. If $d\geq 2$, then the following estimates hold:
\begin{equation}
    \forall u\in H^n(\partition)\quad \exists u_{\mathrm{I}}\in\Wh:\quad|||u-u_{\mathrm{I}}|||^2_{\mathrm{DG}} \lesssim \sum_{K\in\partition} h_K^{2\min\{p+1,n\}-2} ||u||^2_{H^n(K)}.
\end{equation}
\end{proposition}
For detailed proof of the proposition see \cite{houston:book} (for a large enough penalty $\eta_0$). In this section, we assume that problem \eqref{eq:DGFormulation} is supplemented by the initial condition $c_h^0=c_\mathrm{I}(0)\in\Wh$, provided that $c_0(x)$ is sufficiently regular. In this case, we need $c_0\in W$ to interpolate the solution.
\par
First of all, let us consider $c_h$ solution of \eqref{eq:DGFormulation} and $c$ solution of \eqref{eq:weakform}. To extend the bilinear forms of \eqref{eq:DGFormulation} to the space of continuous solutions we need further regularity requirements. We assume element-wise $H^2$-regularity of the concentration together with the continuity of the flow across the interfaces $F\in\facesinternal$ for all time $t\in(0,T]$. In this context, we need to provide additional boundedness results for the functionals of the formulation:
\begin{proposition}
Let Assumption 2 be satisfied. Then:
\begin{equation}
    \exists {M}>0 \quad |\mathcal{A}(u,w_h)| \leq {M} |||u|||_\mathrm{DG} ||w_h||_\mathrm{DG} \qquad \forall u\in  H^2(\mathcal{T}_h),\forall w_h\in \Wh.
\end{equation}
\end{proposition}
The proof of this relation can be found in \cite{Cangiani:PolyDG}. In order to prove the convergence estimate we assume both $f=0$. These assumptions allow us to use the boundedness property of the solution for the initial condition $c_0(\boldsymbol{x})\in(0,1)$ for each $\boldsymbol{x}\in \Omega$ \cite{salsa:EDP}:
\begin{equation}
    ||c(t)||_{L^\infty(\Omega)}^2 \leq 1 \qquad \forall t\in(0,T)
    \label{eq:linfbound}
\end{equation}
\par
Interpolating this type of solution, we can have a function $c_\mathrm{I}$, which is $L^\infty$ by construction \cite{quarteroni:EDP}, then:
\begin{equation}
    \exists M_\mathrm{I}>0:\qquad||c_\mathrm{I}(t)||_{L^\infty(\Omega)}^2 \leq M_\mathrm{I} \qquad \forall t\in(0,T)
    \label{eq:linfboundint}
\end{equation}
\begin{theorem}
Let us consider problem \eqref{eq:weakform} with $f=0$, $\phi_\mathrm{N}=0$ and $\Gamma_N = \partial\Omega$. Let Assumptions 1 and 2 be fulfilled and let c be the solution of \eqref{eq:weakform} for any $t\in(0,T]$ and let it satisfy the following additional regularity requirements:
\begin{equation}
    c\in C^1((0,T]; H^n(\Omega)\cap L^\infty(\Omega)),
\end{equation}
for $n\geq 2$. Let us assume further regularity on the initial condition $c_0\in W$. For a sufficiently large penalty parameter $\eta$, let $c_h$ be the solution of \eqref{eq:DGFormulation} for any $t\in(0,T]$. Then, the following estimate holds:
\begin{equation}
|||c(t)-c_h(t)|||^2_\epsilon \lesssim\sum_{K\in\partition} h_K^{2\min\{p+1,n\}-2} 
\int_0^t \left[||\dot{c}(s)||^2_{H^n(K)}+ ||c(s)||^2_{H^n(K)}\right] \qquad \forall t\in(0,T],
\end{equation}
under the following additional hypothesis of the constants: $\mu \widetilde{C}-\widetilde{\alpha}((1+M_\mathrm{I})C_{E_2}+C_SC_{E_4})>0$, where $C_{E_q}$ is the discrete Sobolev embedding constant for the $L^q(\Omega)$ space, $\widehat{C}$ is the bounding constant between the DG-norms and $C_S$ is defined in \eqref{eq:simplifiedenergyestimate}.
\end{theorem}
\begin{proof}
First of all, we subtract Equation \eqref{eq:DGFormulation} from Equation \eqref{eq:weakform}, to obtain:
\begin{equation*}
     \left(\dot{c}-\dot{c}_h,w_h\right)_\Omega + \mathcal{A}(c-c_h,w_h) - r_L(c-c_h,w_h) + (\alpha (c^2 - c_h^2),w_h)_\Omega = 0 \qquad \forall w_h\in \Wh.
\end{equation*}
We define the errors $e^c_h = c_\mathrm{I}-c_h$ and $e^c_\mathrm{I} = c-c_\mathrm{I}$, where $c_\mathrm{I}$ is a suitable interpolant. By testing against $e^c_h$, we have:
\begin{equation*}
\left(\dot{e}^c_h,e^c_h\right)_\Omega + \mathcal{A}(e^c_h,e^c_h) - r_L(e^c_h,e^c_h) + (\alpha (c^2 - c_h^2),e^c_h)_\Omega = \left(\dot{e}^c_\mathrm{I},e^c_h\right)_\Omega + \mathcal{A}(e^c_\mathrm{I},e^c_h) - r_L(e^c_\mathrm{I},e^c_h).
\end{equation*}
Thanks to the symmetry of the scalar product we can rewrite the problem as:
\begin{equation*}
\dfrac{1}{2}\dfrac{\mathrm{d}}{\mathrm{d}t}\left(e^c_h,e^c_h\right)_\Omega + \mathcal{A}(e^c_h,e^c_h) - r_L(e^c_h,e^c_h) + (\alpha (c^2 - c_h^2),e^c_h)_\Omega = \left(\dot{e}^c_\mathrm{I},e^c_h\right)_\Omega + \mathcal{A}(e^c_\mathrm{I},e^c_h) - r_L(e^c_\mathrm{I},e^c_h).
\end{equation*}
Now we integrate between $0$ and $t$. We remark that $e^c_h(0)=0$ under the suitable choice we made on $c_h^0$. Then, by proceeding similarly to what we did in the proof of Theorem 1, we obtain:
\begin{equation*}
\begin{split}
\dfrac{1}{2}||e^c_h(t)||^2 + & \int_0^t\mu||e^c_h(s)||_\mathrm{DG}^2 \leq \int_0^t \widetilde{\alpha}||e^c_h(s)||^2 + \int_0^t |\left(\dot{e}^c_\mathrm{I}(s),e^c_h(s)\right)_\Omega| \\ + & \int_0^t|\mathcal{A}(e^c_\mathrm{I}(s),e^c_h(s))| + \int_0^t |r_L(e^c_\mathrm{I}(s),e^c_h(s))| + \int_0^t |(\alpha (c^2(s) - c_h^2(s)),e^c_h(s))_\Omega|.
\end{split}
\end{equation*}
In this way, we obtain four different scalar products, we need to bound. Exploiting the continuity relation in Proposition 4,  H\"{o}lder's inequality and $L^\infty$-bound of the parameter $\alpha$ ($\tilde{\alpha} = ||\alpha||_{L^\infty}$), we get:
\begin{equation*}
\begin{split}
\dfrac{1}{2}||e^c_h(t)||^2 + & \int_0^t\mu||e^c_h(s)||_{\mathrm{DG}}^2  \leq \int_0^t \widetilde{\alpha}||e^c_h(s)||^2 + \int_0^t ||\dot{e}^c_\mathrm{I}(s)||\,||e^c_h(s)|| + \int_0^t M|||e^c_\mathrm{I}(s)|||_\mathrm{DG} ||e^c_h(s)||_\mathrm{DG} \\ + & \int_0^t \widetilde{\alpha}||e^c_\mathrm{I}(s)||\,||e^c_h(s)|| + \int_0^t \widetilde{\alpha}|( c^2(s) - c_h^2(s),e^c_h(s))_\Omega|,
\end{split}
\end{equation*}
We treat now the nonlinear term by rewriting the difference as follows:
\begin{equation*}
\begin{split}
c^2-c_h^2 = & c^2 - c_\mathrm{I}^2 + c_\mathrm{I}^2 - c_h^2 \\
= & c^2 - c\,c_\mathrm{I} + c\,c_\mathrm{I} - c_\mathrm{I}^2 + c_\mathrm{I}^2  - c_\mathrm{I}c_h + c_\mathrm{I}c_h - c_h^2 \\
= & c (c - \,c_\mathrm{I}) + c_\mathrm{I} (c - c_\mathrm{I}) + c_\mathrm{I} (c_\mathrm{I} - c_h) + c_h(c_\mathrm{I} - c_h) = \underset{(\mathrm{I})}{\underbrace{c e_\mathrm{I}^c}} + \underset{(\mathrm{II})}{\underbrace{c_\mathrm{I} e_\mathrm{I}^c}}  + \underset{(\mathrm{III})}{\underbrace{c_\mathrm{I} e_h^c}}  + \underset{(\mathrm{IV})}{\underbrace{c_h e_h^c}}.
\end{split}
\end{equation*}
The resulting terms can be treated separately as follows:
\begin{itemize}
    \item $(\mathrm{I})$ can be bounded using the $L^\infty$-bound of the continuous solution, Equation \eqref{eq:linfbound} and the Cauchy-Schwarz inequality:
    \begin{equation*}
    |(c(c-c_\mathrm{I}),e^c_h)_\Omega| \leq ||c||_{L^\infty(\Omega)}|(e^c_\mathrm{I},e^c_h)_\Omega| = |(e^c_\mathrm{I},e^c_h)_\Omega| \leq ||e^c_\mathrm{I}||\;||e^c_h||.
    \end{equation*}
    \item  $(\mathrm{II})$ can be bounded using the $L^\infty$-bound of the interpolant of the continuous solution, Equation \eqref{eq:linfboundint} and the Cauchy-Schwarz inequality:
    \begin{equation*}
    |(c_\mathrm{I}(c-c_\mathrm{I}),e^c_h)_\Omega| \leq ||c_\mathrm{I}||_{L^\infty(\Omega)}|(e^c_\mathrm{I},e^c_h)_\Omega| \leq  M_\mathrm{I} |(e^c_\mathrm{I},e^c_h)_\Omega| \leq  M_\mathrm{I}  ||e^c_\mathrm{I}||\;||e^c_h||.
    \end{equation*}
    \item  $(\mathrm{III})$ can be bounded using the $L^\infty$-bound of the interpolant of the continuous solution, Equation \eqref{eq:linfboundint}, Equation \eqref{eq:simplifiedenergyestimate}, and the Sobolev–Poincaré–Wirtinger discrete inequality \cite{dipietro:HHO}:
    \begin{equation*}
    |(c_\mathrm{I}(c_\mathrm{I}-c_h),e^c_h)_\Omega| \leq ||c_\mathrm{I}||_{L^\infty(\Omega)}|(e^c_h,e^c_h)_\Omega| \leq  M_\mathrm{I} |(e^c_h,e^c_h)_\Omega| =  M_\mathrm{I} ||e^c_h||^2 \leq  M_\mathrm{I} C_{E_2} ||e^c_h||_\mathrm{DG}^2,
    \end{equation*}
    where $C_{E_2}$ is the bounding constant of Sobolev–Poincaré–Wirtinger discrete inequality.
    \item  $(\mathrm{IV})$ can be bounded using H\"{o}lder inequality, the energy stability estimate of the DG solution in Equation \eqref{eq:simplifiedenergyestimate}, and the Sobolev–Poincaré–Wirtinger discrete inequality \cite{dipietro:HHO}
    \begin{equation*}
    |(c_h(c_\mathrm{I}-c_h),e^c_h)_\Omega| = |(c_h,(e^c_h)^2)_\Omega| \leq ||c_h||\;||e^c_h||_{L^4(\Omega)}^2 \leq C_S||e^c_h||_{L^4(\Omega)}^2 \leq C_S C_{E_4} ||e^c_h||_\mathrm{DG}^2.
    \end{equation*}
    where $C_{E_4}$ is the bounding constant of Sobolev–Poincaré–Wirtinger discrete inequality and $C_S$ is defined in Equation \eqref{eq:simplifiedenergyestimate}.
\end{itemize}
Then, from above bounds and by using also the property of DG-norms $\widetilde{C}|||v|||_\mathrm{DG}^2\leq||v||_\mathrm{DG}^2$ we can write:
\begin{equation*}
\begin{split}
\dfrac{1}{2}||e^c_h(t)||^2 + & \int_0^t\mu \widetilde{C}|||e^c_h(s)|||_{\mathrm{DG}}^2  \leq \int_0^t \widetilde{\alpha}((1+M_\mathrm{I})C_{E_2}+C_SC_{E_4})|||e^c_h(s)|||_{\mathrm{DG}}^2 + \int_0^t ||\dot{e}^c_\mathrm{I}(s)||\,||e^c_h(s)||   \\ + & \int_0^t M|||e^c_\mathrm{I}(s)|||_\mathrm{DG} ||e^c_h(s)||_\mathrm{DG} + \int_0^t \widetilde{\alpha}(2+M_\mathrm{I})||e^c_\mathrm{I}(s)||\,||e^c_h(s)||,
\end{split}
\end{equation*}
By assumption, we need $\mu \widetilde{C}-\widetilde{\alpha}((1+M_\mathrm{I})C_{E_2}+C_SC_{E_4})>0$, then we can define $\hat{C} = \min\{1/2,\mu \widetilde{C}-\widetilde{\alpha}((1+M_\mathrm{I})C_{E_2}+C_SC_{E_4})\}$. Since $\hat{C}$ is positive we can make use of the notation $\lesssim$:
\begin{equation*}
|||e^c_h(t)|||^2_\epsilon \lesssim \int_0^t ||\dot{e}^c_\mathrm{I}(s)||\,||e^c_h(s)|| + \int_0^t|||e^c_\mathrm{I}(s)|||_\mathrm{DG} ||e^c_h(s)||_\mathrm{DG} + \int_0^t ||e^c_\mathrm{I}(s)||\,||e^c_h(s)||,
\end{equation*}

By application of H\"{o}lder's inequality and of Gr\"onwall's lemma \cite{quarteroni:EDP}, we obtain:
\begin{equation*}
    |||e^c_h(t)|||^2_\epsilon \lesssim \int_0^t||\dot{e}^c_\mathrm{I}(s)||^2 + \int_0^t|||e^c_\mathrm{I}(s)|||_\mathrm{DG}^2,
\end{equation*}
and by using the interpolation bounds of Proposition 3, we find:
\begin{equation}
|||e^c_h(t)|||^2_\epsilon \lesssim\sum_{K\in\partition} h_K^{2\min\{p+1,n\}-2} 
\int_0^t \left(||\dot{c}(s)||^2_{H^n(K)}+ ||c(s)||^2_{H^n(K)}\right).
\label{eq:resproof}
\end{equation}
Finally, we use the triangular inequality to estimate the discretization error.
\begin{equation*}
    |||c-c_h|||_\epsilon^2 \leq  |||e^c_h|||_\epsilon^2 + |||e^c_\mathrm{I}|||_\epsilon^2
\end{equation*}
The thesis follows by applying the result in Equation \eqref{eq:resproof} and the interpolation error.
\end{proof}
\begin{remark}
    So far, our analysis was based on the assumption of time-independent physical parameters $\alpha=\alpha(\boldsymbol{x})$ and $\mathbf{D}=\mathbf{D}(\boldsymbol{x})$. The results however remain valid also in the case of time-dependent parameters assuming sufficient regularity on time.
\end{remark}
\begin{remark}
    The theoretical analysis proposed in this work is specifically constructed for the FK equation. Indeed, some steps cannot be extended to different types of nonlinear reaction terms. For generalized results on general semilinear parabolic problems we refer to \cite{suli:parabolic}.
\end{remark}
\begin{remark}
The extensions to the non-homogeneous Dirichlet/Neumann boundary conditions can be proved by assuming sufficient regularity on the data, as in the linear case \cite{cangiani_hp-version_2016}. The results can be also extended to Robin boundary conditions, as in the linear case, changing the formulation and the proof as in \cite{riviereDiscontinuousGalerkinMethods2008}. Concerning Theorem~2, the extension of the theoretical result under different boundary conditions does not allow the use of Equation \eqref{eq:linfbound} in the continuous setting. It requires assuming a continuous weak solution $c\in L^{\infty}(\Omega)$ by taking care of defining the value $M$, such that $\|c\|_{L^\infty(\Omega)}\leq M$, and modifying the proof accordingly.
\end{remark}

\section{Fully-discrete formulation}
\label{sec:temporaldisc}
Let $(\boldsymbol\varphi_j)_{j=0}^{N_c}$ be a suitable basis for $\Wh$, then $c_h(t) = \displaystyle\sum_{j=0}^{N_c} C_n(t)\varphi_j,$ and denote by $\boldsymbol{C}\in\mathbb{R}^{N_c}$ the corresponding vector of the expansion coefficients, in the chosen basis. We define the matrices:
\begin{equation*}
[\mathrm{M}]_{ij} = (\varphi_j, \varphi_i)_\Omega\quad\mathrm{(Mass\;matrix)} \qquad 
[\mathrm{A}]_{ij} = \mathcal{A}(\varphi_j, \varphi_i)\quad\mathrm{(Stiffness\;matrix)}\qquad I,j = 1,...,N_c
\end{equation*}
\smallskip
\begin{equation*}
[\mathrm{M}_\alpha]_{ij} = (\alpha \varphi_j, \varphi_i)_\Omega\quad\mathrm{(Linear\;reaction\;matrix)} \qquad
[\widetilde{\mathrm{M}}_\alpha(\boldsymbol{C}(t)) ]_{ij}= (\alpha c_h(t)\varphi_j, \varphi_i)_\Omega\quad\mathrm{(Nonlinear\;reaction\;matrix)}
\end{equation*}
Moreover, we define the forcing term: $[\boldsymbol{F}]_{j} = F(\varphi_j)$ for $j = 1,...,N_c$. By exploiting all these definitions, we rewrite the problem \eqref{eq:DGFormulation} in algebraic form:
\begin{equation}
\label{eq:algfull}
    \begin{dcases}
         \mathrm{M}\dot{\boldsymbol{C}}(t)+\mathrm{A}\boldsymbol{C}(t)-\mathrm{M}_\alpha\boldsymbol{C}(t) + \widetilde{\mathrm{M}}_\alpha (\boldsymbol{C}(t))\boldsymbol{C}(t) = \boldsymbol{F}(t), & t\in(0,T) \\[4pt]
         \boldsymbol{C}(0) = \boldsymbol{C}_0 
    \end{dcases}
\end{equation}
\par
Let now construct a time discretization of the interval $[0,T]$ by defining a partition of $N$ intervals $0=t_0<t_1<...<t_N=T$. We assume a constant timestep $\Delta t = t_{n+1}-t_n$, $n=0,...,N-1$. We construct the fully discrete approximation by means of the Crank-Nicolson method. Given $\boldsymbol{C}^{0}=\boldsymbol{C}(0)$, find $\boldsymbol{C}^{n+1}\simeq \boldsymbol{C}(t_{n+1})$ for $n=0,...,N-1$:
\begin{equation}
\label{eq:fullyalgebraicproblem}
\mathrm{M}\boldsymbol{C}^{n+1}+\dfrac{\Delta t}{2}\left(\mathrm{K}-\mathrm{M}_\alpha\right)\boldsymbol{C}^{n+1} + \Delta t \widetilde{\mathrm{M}}_\alpha^{1/2} (\boldsymbol{C}^*)\boldsymbol{C}^{n+1,n} = \mathrm{M}\boldsymbol{C}^{n}-\dfrac{\Delta t}{2}\left(\mathrm{K}-\mathrm{M}_\alpha \right)\boldsymbol{C}^{n}+ \dfrac{1}{2}\left(\boldsymbol{F}^{n+1} + \boldsymbol{F}^n\right).
\end{equation}
For the nonlinear term we will consider either:
\begin{itemize}
    \item Semi-implicit treatment, i.e.:
    \begin{equation}
        \tilde{\mathrm{M}}_\alpha^{1/2} \left(\frac{3}{2}\boldsymbol{C}^n-\frac{1}{2}\boldsymbol{C}^{n-1}\right)\;\dfrac{\boldsymbol{C}^{n+1}+\boldsymbol{C}^{n}}{2},
    \end{equation}
\item Implicit treatment, i.e.:
    \begin{equation}
        \tilde{\mathrm{M}}_\alpha^{1/2} \left(\frac{1}{2}\boldsymbol{C}^{n+1}+\frac{1}{2}\boldsymbol{C}^{n}\right)\;\dfrac{\boldsymbol{C}^{n+1}+\boldsymbol{C}^{n}}{2}.
    \end{equation}
\end{itemize}

\section{Numerical results}
\label{sec:numericalresults}
In this section, we aim at verifying the accuracy of the method and the theoretical bounds of Section \ref{sec:error}. Throughout the section we choose the penalty parameter $\eta_0=10$.

\subsection{Test case 1: convergence analysis in a 2D case}
\begin{figure}[t!]
    \begin{subfigure}[b]{0.5\textwidth}
          \resizebox{\textwidth}{!}{\definecolor{mycolor2}{rgb}{0.00000,1.00000,1.00000}%
\pgfplotsset{
  log x ticks with fixed point/.style={
      xticklabel={
        \pgfkeys{/pgf/fpu=true}
        \pgfmathparse{exp(\tick)}%
        \pgfmathprintnumber[fixed  zerofill, precision=2]{\pgfmathresult}
        \pgfkeys{/pgf/fpu=false}
      }
  }
}
\begin{tikzpicture}

\begin{axis}[%
width=3.875in,
height=2.36in,
at={(2.6in,1.099in)},
scale only axis,
xmode=log,
xmin=0.064,
xmax=0.3239,
xminorticks=true,
xlabel = {$h$ [-]},
ylabel = {$||c(T)-c_h(T)||_\varepsilon$},
ymode=log,
ymin=1e-12,
ymax=0.01,
yminorticks=true,
axis background/.style={fill=white},
title style={font=\bfseries},
title={Semi-implicit time-discretization scheme},
xmajorgrids,
xminorgrids,
ymajorgrids,
yminorgrids,
legend style={legend cell align=left, align=left, draw=white!15!black}
]

\addplot [color=red, line width=2.0pt]
  table[row sep=crcr]{%
0.326976567914328  0.005602751537762\\
0.182774579337139  0.002306018088957\\
0.109324383591326  0.001299928383272\\
0.064564291840915  0.000572541903204\\
};
\addlegendentry{$p=1$}

\addplot [color=orange, line width=2.0pt]
  table[row sep=crcr]{%
0.328217726958287	0.003664756267069\\
0.188895143477632	0.001117233947070\\
0.108839088345096	3.70577524791e-04\\
0.064974829517436	0.000112266948587\\
};
\addlegendentry{$p=2$}

\addplot [color=green, line width=2.0pt]
  table[row sep=crcr]{%
0.324901291854012	4.512504331807238e-04\\
0.191027304794974   6.629507050594746e-05\\
0.112174284034378	1.158680994896928e-05\\
0.062496397121138	0.001858866712376e-03\\
};
\addlegendentry{$p=3$}

\addplot [color=mycolor2, line width=2.0pt]
  table[row sep=crcr]{%
0.3255	4.4357e-05\\
0.1816	3.7904e-06\\
0.1085	3.8867e-07\\
0.0618	3.4081e-08\\
};
\addlegendentry{$p=4$}

\addplot [color=blue, line width=2.0pt]
  table[row sep=crcr]{%
0.3255	2.2868e-06\\
0.1816	9.5398e-08\\
0.1085	5.3892e-09\\
0.0618  2.5799e-10\\
};
\addlegendentry{$p=5$}

\addplot [color=purple, line width=2.0pt]
  table[row sep=crcr]{%
0.319866273494550	1.223910310572353e-07\\
0.191962900117614	3.267462944072783e-09\\
0.108354266764863	1.069354535214168e-10\\
0.0639	            5.4727e-12\\
};
\addlegendentry{$p=6$}

\node[right, align=left, text=black, font=\footnotesize]
at (axis cs:0.1005,0.00075) {$1$};

\addplot [color=black, line width=1.5pt]
  table[row sep=crcr]{%
0.100   0.0008\\
0.075   0.0006\\
0.100   0.0006\\
0.100   0.0008\\
};

\node[right, align=left, text=black, font=\footnotesize]
at (axis cs:0.1005,0.0002) {$2$};

\addplot [color=black, line width=1.5pt]
  table[row sep=crcr]{%
0.100   0.0002\\
0.075   0.0001125\\
0.100   0.0001125\\
0.100   0.0002\\
};

\node[right, align=left, text=black, font=\footnotesize]
at (axis cs:0.1005,4e-6) {$3$};

\addplot [color=black, line width=1.5pt]
  table[row sep=crcr]{%
0.100   6e-06\\
0.075   2.53e-06\\
0.100   2.53e-06\\
0.100   6e-06\\
};

\node[right, align=left, text=black, font=\footnotesize]
at (axis cs:0.1005,1e-7) {$4$};

\addplot [color=black, line width=1.5pt]
  table[row sep=crcr]{%
0.100   2e-07\\
0.075   6.328e-08\\
0.100   6.328e-08\\
0.100   2e-07\\
};

\node[right, align=left, text=black, font=\footnotesize]
at (axis cs:0.1005,1e-9) {$5$};

\addplot [color=black, line width=1.5pt]
  table[row sep=crcr]{%
0.100   2e-09\\
0.075   4.74e-10\\
0.100   4.74e-10\\
0.100   2e-09\\
};

\node[right, align=left, text=black, font=\footnotesize]
at (axis cs:0.1005,2e-11) {$6$};

\addplot [color=black, line width=1.5pt]
  table[row sep=crcr]{%
0.100   5e-11\\
0.075   8.8989e-12\\
0.100   8.8989e-12\\
0.100   5e-11\\
};

\end{axis}
\end{tikzpicture}%
}
    \end{subfigure}%
    \begin{subfigure}[b]{0.5\textwidth}
        \resizebox{\textwidth}{!}{\definecolor{mycolor2}{rgb}{0.00000,1.00000,1.00000}%
\pgfplotsset{
  log x ticks with fixed point/.style={
      xticklabel={
        \pgfkeys{/pgf/fpu=true}
        \pgfmathparse{exp(\tick)}%
        \pgfmathprintnumber[fixed  zerofill, precision=2]{\pgfmathresult}
        \pgfkeys{/pgf/fpu=false}
      }
  }
}
\begin{tikzpicture}

\begin{axis}[%
width=3.875in,
height=2.36in,
at={(2.6in,1.099in)},
scale only axis,
xmode=log,
xmin=0.064,
xmax=0.3239,
xminorticks=true,
xlabel = {$h$ [-]},
ylabel = {$||c(T)-c_h(T)||_\varepsilon$},
ymode=log,
ymin=1e-12,
ymax=0.01,
yminorticks=true,
axis background/.style={fill=white},
title style={font=\bfseries},
title={Implicit time-discretization scheme},
xmajorgrids,
xminorgrids,
ymajorgrids,
yminorgrids,
legend style={legend cell align=left, align=left, draw=white!15!black}
]
              
\addplot [color=red, line width=2.0pt]
  table[row sep=crcr]{%
0.341382280968322  0.005624830953481\\
0.190172648607784  0.002390432407803\\
0.108957016308869  0.001170357213753\\
0.064281923221838  0.000568483703791\\
};

\addplot [color=orange, line width=2.0pt]
  table[row sep=crcr]{%
0.347443616352793   0.003698872140207\\
0.182273186961405	0.001047846889660\\
0.109135957842343	3.59999188428e-04\\
0.064974829517436	1.1174e-04\\
};

\addplot [color=green, line width=2.0pt]
  table[row sep=crcr]{%
0.335716266752722	4.591700912889362e-04\\
0.178599013975622   6.564654434689503e-05\\
0.111756783886366	1.171957579094479e-05\\
0.0639	            1.9103e-06\\
};

\addplot [color=mycolor2, line width=2.0pt]
  table[row sep=crcr]{%
0.321622255288620   4.446612930731537e-05\\
0.183859192450264	3.835283934756732e-06\\
0.111170933166923	3.813209704260621e-07\\
0.062917536881006	3.159627779690000e-08\\
};

\addplot [color=blue, line width=2.0pt]
  table[row sep=crcr]{%
0.3255	2.2868e-06\\
0.1816	9.5398e-08\\
0.1085	5.3892e-09\\
0.0618	2.5799e-10\\
};

\addplot [color=purple, line width=2.0pt]
  table[row sep=crcr]{%
0.3255	1.3519e-07\\
0.1816	3.3333e-09\\
0.1085	1.1063e-10\\
0.0639	5.6243e-12\\
};

\node[right, align=left, text=black, font=\footnotesize]
at (axis cs:0.1005,0.00075) {$1$};

\addplot [color=black, line width=1.5pt]
  table[row sep=crcr]{%
0.100   0.0008\\
0.075   0.0006\\
0.100   0.0006\\
0.100   0.0008\\
};

\node[right, align=left, text=black, font=\footnotesize]
at (axis cs:0.1005,0.0002) {$2$};

\addplot [color=black, line width=1.5pt]
  table[row sep=crcr]{%
0.100   0.0002\\
0.075   0.0001125\\
0.100   0.0001125\\
0.100   0.0002\\
};

\node[right, align=left, text=black, font=\footnotesize]
at (axis cs:0.1005,4e-6) {$3$};

\addplot [color=black, line width=1.5pt]
  table[row sep=crcr]{%
0.100   6e-06\\
0.075   2.53e-06\\
0.100   2.53e-06\\
0.100   6e-06\\
};

\node[right, align=left, text=black, font=\footnotesize]
at (axis cs:0.1005,1e-7) {$4$};

\addplot [color=black, line width=1.5pt]
  table[row sep=crcr]{%
0.100   2e-07\\
0.075   6.328e-08\\
0.100   6.328e-08\\
0.100   2e-07\\
};

\node[right, align=left, text=black, font=\footnotesize]
at (axis cs:0.1005,1e-9) {$5$};

\addplot [color=black, line width=1.5pt]
  table[row sep=crcr]{%
0.100   2e-09\\
0.075   4.74e-10\\
0.100   4.74e-10\\
0.100   2e-09\\
};

\node[right, align=left, text=black, font=\footnotesize]
at (axis cs:0.1005,2e-11) {$6$};

\addplot [color=black, line width=1.5pt]
  table[row sep=crcr]{%
0.100   5e-11\\
0.075   8.8989e-12\\
0.100   8.8989e-12\\
0.100   5e-11\\
};

\end{axis}
\end{tikzpicture}%
}
    \end{subfigure}%
    \caption{Test case 1: computed errors and convergence rates with either semi-implicit (left) and implicit (right) treatment of the nonlinear term.}
    \label{fig:errors2D}
\end{figure}
\begin{figure}
    \begin{subfigure}[b]{0.5\textwidth}
          \resizebox{\textwidth}{!}{\definecolor{mycolor1}{rgb}{1.00000,1.00000,0.00000}%
\definecolor{mycolor2}{rgb}{0.00000,1.00000,1.00000}%

\begin{tikzpicture}
\begin{axis}[%
width=3.875in,
height=2in,
at={(1.733in,0.687in)},
scale only axis,
xmin=1,
xmax=8,
xlabel style={font=\color{white!15!black}},
xlabel={$p$},
ymode=log,
ymin=1e-10,
ymax=0.01,
yminorticks=true,
ylabel style={font=\color{white!15!black}},
ylabel={Error},
axis background/.style={fill=white},
title={Semi-implicit treatment of the nonlinear term},
xmajorgrids,
ymajorgrids,
yminorgrids,
legend style={legend cell align=left, align=left, draw=white!15!black}
]

\addplot [color=magenta, line width=2.0pt]
  table[row sep=crcr]{%
1	5.60e-03\\
2	3.60e-03\\
3	4.50e-04\\
4	4.40e-05\\
5   2.28e-06\\
6	1.22e-07\\
7   5.41e-09\\
8   2.27e-10\\
};
\addlegendentry{$||c(T)-c_h(T)||_\varepsilon$}

\end{axis}
\end{tikzpicture}%
}
    \end{subfigure}%
    \begin{subfigure}[b]{0.5\textwidth}
        \resizebox{\textwidth}{!}{\definecolor{mycolor1}{rgb}{1.00000,1.00000,0.00000}%
\definecolor{mycolor2}{rgb}{0.00000,1.00000,1.00000}%

\begin{tikzpicture}
\begin{axis}[%
width=3.875in,
height=2in,
at={(1.733in,0.687in)},
scale only axis,
xmin=1,
xmax=8,
xlabel style={font=\color{white!15!black}},
xlabel={$p$},
ymode=log,
ymin=1e-10,
ymax=0.01,
yminorticks=true,
ylabel style={font=\color{white!15!black}},
ylabel={Error},
axis background/.style={fill=white},
title={Implicit treatment of the nonlinear term},
xmajorgrids,
ymajorgrids,
yminorgrids,
legend style={legend cell align=left, align=left, draw=white!15!black}
]

\addplot [color=cyan, line width=2.0pt]
  table[row sep=crcr]{%
1	5.60e-03\\
2	3.50e-03\\
3	4.90e-04\\
4	4.44e-05\\
5   2.28e-06\\
6	1.35e-07\\
7   5.04e-09\\
8   3.41e-10\\
};
\addlegendentry{$||c(T)-c_h(T)||_\varepsilon$}

\end{axis}
\end{tikzpicture}%
}
    \end{subfigure}%
    \caption{Test case 1: computed errors and convergence rates with either semi-implicit (left) and implicit (right) treatment of the nonlinear term.}
    \label{fig:errors2Dpoly}
\end{figure}
For the numerical tests in this section, we use a MATLAB code to solve the FK equation on polygonal meshes. We use a square domain $\Omega=(0,1)^2$, where we construct a mesh by using PolyMesher \cite{Polymesher}. Concerning the time discretization, we use a timestep $\Delta t = 10^{-5}$ and a maximum time $T=10^{-3}$. We consider the following manufactured exact solution:
\begin{equation}
    c(x,y,t)=\left(\cos(\pi x)\cos(\pi y)+2\right) e^{-t}.
\end{equation}
A fundamental simplification in this section is the isotropic diffusion tensor $\mathbf{D}=d_\mathrm{ext}\mathbf{I}$. We analyse the case with $d_\mathrm{ext}=1$ and $\alpha=1$. The forcing term and the Dirichlet boundary conditions are derived accordingly. 
\par
In Figure \ref{fig:errors2D}, we report the computed errors in the energy norm defined in Equation \eqref{eq:energynorm} at the final time $T=10^{-3}$. We performed the convergence test keeping fixed the polynomial order of the space approximation $p=1,...,6$ and using different mesh refinements $(N_\mathrm{el}= 30,100,300,1000)$. We observe that the theoretical rates of convergence are achieved for all the polynomial degrees $p$; indeed, the rate of convergence equals the degree of approximation, as proved in Theorem 2. 
\par
In Figure \ref{fig:errors2D}, we compare also the errors in the two different choices of treatment of the nonlinear term: the semi-implicit (left) and the implicit one (right). In the implicit case, the resulting nonlinear problem is solved by means of Picard iterations with tollerace $10^{-10}$. In this test case, we cannot notice any substantial difference concerning the resulting errors and the two methods reach the same precision for all the tested values of $p$.
\par
A convergence analysis with respect to the polynomial order $p$ is also performed with a mesh of 30 elements. The results are reported in Figure \ref{fig:errors2Dpoly}, where we observe exponential convergence. We point out that this case is not covered by our theoretical analysis, nevertheless, we demonstrate numerically that optimal convergence is observed. Also in this case we cannot notice any difference in the choice of the nonlinear treatment.

\subsection{Test case 2: Travelling waves in 2D}
In this section, we use the PolyDG formulation to simulate the travelling-wave solution of the FK equation in 2D:
\begin{equation}
    c(x,y,t) = \psi(x-vt) = \psi(\xi)
    \label{eq:wave}
\end{equation}
By plugging Equation \eqref{eq:wave} into Equation \eqref{eq:fk_strong}, with $f=0$ we obtain an equivalent system of ordinary differential equations:
\begin{equation}
    \begin{dcases}
        \chi'(\xi) = -\dfrac{v}{d_\mathrm{ext}}\chi(\xi) + \dfrac{1}{d_\mathrm{ext}}\psi(\xi)(\psi(\xi)-1) & \xi\in(0,T), \\
        \psi'(\xi) = \chi(\xi) & \xi\in(0,T), \\
    \end{dcases}
    \label{eq:sysode}
\end{equation}
where we use the assumption of isotropic diffusion tensor $\mathbf{D}=d_\mathrm{ext}\mathbf{I}$ and $d_\mathrm{axn}=0$. In particular, we fix $d_\mathrm{ext}=10^{-3}$, $\alpha=1$ and $\eta_0=10$. Concerning the wave's parameters we consider a speed $v=0.1$ and the initial data $\psi(0)=1$ and $\chi(0)=-10^{-2}$. The domain is constructed as a rectangle $\Omega=(0, 5)\times(0, 1)$ and we consider two final times $T=5$ and $T=10$. For the implicit treatment of the nonlinear terms, we adopt Picard iterations with an absolute tolerance $10^{-10}$ and a maximum number of iterations fixed to $20$. The exact solution of this test case is not known in closed form \cite{bonizzoni_structure-preserving_2020}. The reference solution is computed by solving Equation \eqref{eq:sysode} with the MATLAB solver \texttt{ode45}.

\begin{table}[t]
    \centering
    \begin{tabular}{C|c|C|C|C|C}
    \hline
    \multicolumn{5}{c}{$p=2 \qquad \Delta t= 0.01$} \\
    \hline
    \hline
    \multicolumn{2}{c|}{\textbf{Method}} 
    & \multicolumn{2}{c|}{\textbf{Semi-Implicit}} 
    & \multicolumn{2}{c}{\textbf{Implicit}}
    \\    \hline
    \textbf{Refinement} & \textbf{DOFs}
    & $T = 5$ 
    & $T = 10$ 
    & $T = 5$ 
    & $T = 10$
    \\ \hline
    $N_\mathrm{el} = 30$ & $180$
    & $6.33\times10^{3}$ & $1.03\times10^{4}$
    & $1.63\times10^{0}$ & $5.36\times10^{6}$
    \\ \hline
    $N_\mathrm{el} = 100$ & $600$
    & $1.45\times10^{2}$ & $1.12\times10^{4}$
    & $8.24\times10^{-1}$ & $2.18\times10^{7}$
    \\ \hline    
    $N_\mathrm{el} = 300$ & $1800$
    & $1.98\times10^{-1}$ & $6.02\times10^{4}$
    & $6.97\times10^{-2}$ & $1.28\times10^{8}$
    \\ \hline
    \end{tabular}

    \begin{tabular}{C|c|C|C|C|C}
    \hline
    \multicolumn{5}{c}{$p=3 \qquad \Delta t= 0.01$} \\
    \hline
    \hline 
    \multicolumn{2}{c|}{\textbf{Method}}
    & \multicolumn{2}{c|}{\textbf{Semi-Implicit}} 
    & \multicolumn{2}{c}{\textbf{Implicit}}
    \\    \hline
    \textbf{Refinement} & \textbf{DOFs}
    & $T = 5$ 
    & $T = 10$ 
    & $T = 5$ 
    & $T = 10$
    \\ \hline
    $N_\mathrm{el} = 30$ & $300$
    & $9.27\times10^{-1}$ & $6.75\times10^{4}$
    & $1.56\times10^{-1}$ & $7.34\times10^{7}$
    \\ \hline
    $N_\mathrm{el} = 100$ & $1000$
    & $5.50\times10^{-2}$ & $7.20\times10^{-1}$
    & $8.12\times10^{-3}$ & $1.78\times10^{-1}$
    \\ \hline    
    $N_\mathrm{el} = 300$ & $3000$
    & $6.35\times10^{-4}$ & $7.80\times10^{-3}$
    & $7.71\times10^{-4}$ & $2.12\times10^{-3}$
    \\ \hline
    \end{tabular}
    \caption{Computed errors in the $L^2-$norm at the final time with different mesh refinements, $\Delta t = 0.01$: $p=2$ (top) and $p=3$ (bottom).}
    \label{tab:errorsgridref}
\end{table}
\begin{figure}[t]
	\centering
	{\includegraphics[width=\textwidth]{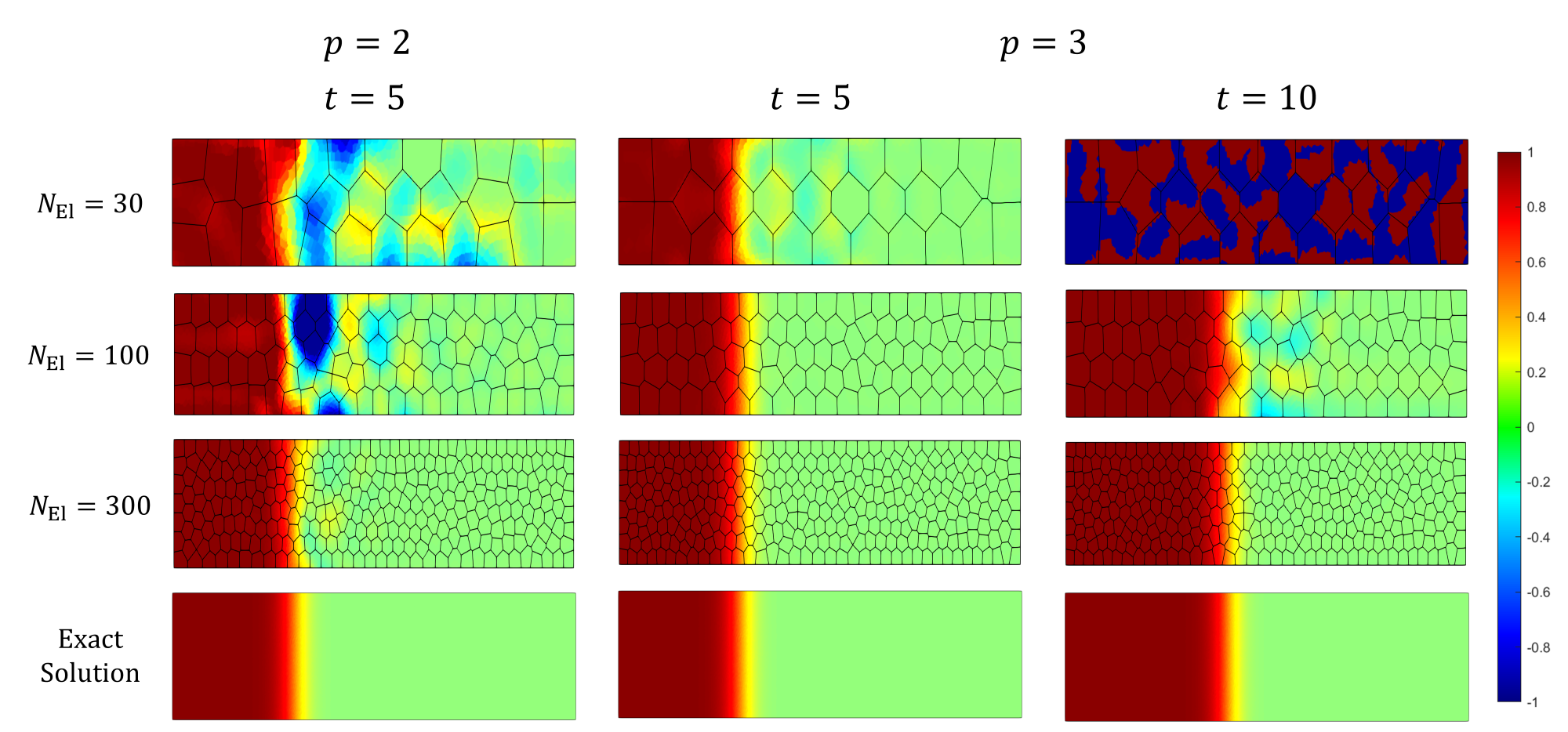}}
	\caption{Snapshot of the exact (last row) and computed (first two rows) solutions with different values of mesh refinement and semi-implicit solver. A correct approximation of the wave propagation velocity can be observed by comparing the last two rows.}
	\label{fig:waves2D}
\end{figure}
\begin{table}[t]
    \centering
    \begin{tabular}{C|c|C|C|C|C}
    \hline
    \multicolumn{5}{c}{$\Delta t = 0.01$} \\
    \hline
    \hline
    \multicolumn{2}{c|}{\textbf{Method}}
    & \multicolumn{2}{c|}{\textbf{Semi-Implicit}} 
    & \multicolumn{2}{c}{\textbf{Implicit}}
    \\    \hline
    \textbf{Order}  & \textbf{DOFs}
    & $T = 5$ 
    & $T = 10$ 
    & $T = 5$ 
    & $T = 10$
    \\ \hline
    $\boldsymbol{p=2}$ & $180$
    & $6.33\times10^{3}$ & $1.03\times10^{4}$
    & $1.63\times10^{0}$ & $5.36\times10^{6}$
    \\ \hline
    $\boldsymbol{p=3}$ & $300$
    & $9.27\times10^{-1}$ & $6.75\times10^{4}$
    & $1.56\times10^{-1}$ & $7.34\times10^{7}$
    \\ \hline    
    $\boldsymbol{p=4}$ & $450$
    & $1.80\times10^{-2}$ & $1.96\times10^{-1}$
    & $8.10\times10^{-3}$ & $9.98\times10^{-2}$
    \\ \hline    
    $\boldsymbol{p=5}$ & $630$
    & $1.40\times10^{-3}$ & $1.80\times10^{-2}$ 
    & $2.25\times10^{-3}$ & $4.54\times10^{-2}$ 
    \\ \hline
    \end{tabular}

    \begin{tabular}{C|c|C|C|C|C}
    \hline
    \multicolumn{5}{c}{$\Delta t = 0.005$} \\
    \hline
    \hline
    \multicolumn{2}{c|}{\textbf{Method}}
    & \multicolumn{2}{c|}{\textbf{Semi-Implicit}} 
    & \multicolumn{2}{c}{\textbf{Implicit}}
    \\    \hline
    \textbf{Order} & \textbf{DOFs}
    & $T = 5$ 
    & $T = 10$ 
    & $T = 5$ 
    & $T = 10$
    \\ \hline
    $\boldsymbol{p=2}$ & $180$
    & $1.34\times10^{0}$ & $3.32\times10^{4}$
    & $8.66\times10^{-1}$ & $6.06\times10^{6}$
    \\ \hline
    $\boldsymbol{p=3}$ & $300$
    & $1.28\times10^{-1}$ & $9.06\times10^{3}$
    & $1.02\times10^{-1}$ & $2.58\times10^{7}$
    \\ \hline     
    $\boldsymbol{p=4}$ & $450$
    & $8.50\times10^{-3}$ & $1.20\times10^{-1}$
    & $7.00\times10^{-3}$ & $1.56\times10^{-1}$
    \\ \hline    
    $\boldsymbol{p=5}$ & $630$
    & $1.40\times10^{-3}$ & $9.60\times10^{-3}$ 
    & $5.25\times10^{-4}$ & $2.31\times10^{-3}$ 
    \\ \hline
    \end{tabular}
    \caption{Computed errors in the $L^2-$norm at the final time with different mesh refinement $N_\mathrm{el}$: $\Delta t = 0.01$ (top) and $\Delta t = 0.005$ (bottom).}
    \label{tab:errorstimeref}
\end{table}

First, we try to address the effect of mesh refinement on the quality of the discrete solution. In Table \ref{tab:errorsgridref}, we report the computed errors in $L^2-$norm for the choice $\Delta t = 0.01$ at the two different time frames $T=5$ and $T=10$. 
\par
Concerning the polynomial degree $p=2$, we notice that the scheme provides a good approximation of the wavefront only considering sufficiently refined mesh. In these numerical experiments, we can observe that the use of Picard iterations (implicit treatment) allows for obtaining better error estimates. In Figure \ref{fig:waves2D}, we can observe the results with the semi-implicit discretization at time $t=5$. However, in semi-implicit and implicit cases we notice large errors at time $T=10$. The numerical solution is no longer a sufficiently accurate approximation of the exact one (see Table~\ref{tab:errorsgridref}). Indeed, due to the unstable nature of the equilibrium $c=0$ and to the fact that our method is not positivity-preserving, whenever the numerical solution becomes negative, the scheme is not able to correct approximate the solution, and it diverges to wrong unphysical approximations.
\par
A way to overcome the problem is to increase the polynomial order of the approximation. Indeed, by choosing $p=3$ the solution is accurately approximated at both $t=5$ and $t=10$, for sufficiently refined meshes (see Figure \ref{fig:waves2D}). In Table \ref{tab:errorsgridref}, we notice that the errors remain low also for $t=10$, excluding the case $N_\mathrm{el}=30$. In this case, we do not notice any advantage in the use of an implicit treatment over a semi-implicit one. 
\par
The second test case addresses the effect of the timestep choice on the quality of the solution. In Table \ref{tab:errorstimeref}, we report the computed errors in $L^2-$norm for the choice $N_\mathrm{el} = 30$ at two different snapshots. The first fact that can be noticed is that by reducing the timestep, we have a reduction in the $L^2$-error.
\par
In this test, we can notice the importance of using a high-order numerical scheme, which allows simulating the waves in an accurate way, also on coarse meshes. For example, by using $N_\mathrm{el} = 30$ and $p=5$ we are able to obtain a good approximation of the solution (630 DOFs), on the contrary, with $N_\mathrm{el} = 100$ and $p=2$ we obtain a worst result with a comparable number of DOFs (600). In Figure \ref{fig:ErrorsDOF} on the left, we plot the computed errors at final time $T=5$ versus DOFs in three different cases and with implicit treatment of the nonlinear term: $h$-refinement with fixed polynomial order $p=2,3$, and $p$-refinement with fixed mesh with $N_\mathrm{el}=30$. We can notice that using a higher polynomial order, we have lower errors with the same number of DOFs and use an $h$-refinement strategy. This is coherent with the literature findings about wave simulations \cite{antoniettibottimazzieriwaves}. The test does not evidence large differences in the use of an implicit solver, but this is in general more accurate than the semi-implicit one. However, the resolution with implicit nonlinear treatment requires performing Picard iterations at any timestep and so it requires a higher computational cost.
\par
In Figure \ref{fig:ErrorsDOF} on the right, we report the errors in the energy norm in three different cases associated with different polynomial orders $(p=2,3,4)$, versus time. For this test, we consider $N_\mathrm{el}=30$ and $\Delta t= 0.01$. From these results, it seems that for $p=3,4$ the error increase linearly with $T$, whereas for the case $p=2$, we can observe an exponential trend, after $T=4.5$, which is coherent to the result of Theorem 1.
\par
Finally, in this numerical test, we can observe that with a sufficiently refined mesh and a polynomial order which is large enough, we are able to accurately simulate the wave propagation. In particular, in Figure \ref{fig:waves2D} (last two lines), we can also notice that the velocity of the propagating front is correctly caught by our method. This analysis is fundamental to confirm the accuracy of our method for the prediction of the spreading of the protein concentrations inside the brain.

\begin{figure}
    \begin{subfigure}[b]{0.5\textwidth}
          \resizebox{\textwidth}{!}{\definecolor{mycolor2}{rgb}{0.00000,1.00000,1.00000}%
\pgfplotsset{
  log x ticks with fixed point/.style={
      xticklabel={
        \pgfkeys{/pgf/fpu=true}
        \pgfmathparse{exp(\tick)}%
        \pgfmathprintnumber[fixed  zerofill, precision=2]{\pgfmathresult}
        \pgfkeys{/pgf/fpu=false}
      }
  }
}
\begin{tikzpicture}

\begin{axis}[%
width=3.875in,
height=2in,
at={(1.733in,0.687in)},
scale only axis,
xmode=log,
xmin=100,
xmax=4000,
xminorticks=true,
xlabel = {DOFs [-]},
ylabel = {$||c(T)-c_h(T)||_\varepsilon$},
ymode=log,
ymin=1e-4,
ymax=2,
yminorticks=true,
axis background/.style={fill=white},
title style={font=\bfseries},
xmajorgrids,
xminorgrids,
ymajorgrids,
yminorgrids,
legend style={legend cell align=left, align=left, draw=white!15!black}
]

\addplot [color=orange, line width=2.0pt]
  table[row sep=crcr]{%
180  1.63\\
600  8.24e-1\\
1800 6.97e-2\\
};
\addlegendentry{$h$-refinement ($p=2$)}

\addplot [color=green, line width=2.0pt]
  table[row sep=crcr]{%
300   1.56e-1\\
1000  8.12e-3\\
3000  7.71e-4\\
};
\addlegendentry{$h$-refinement ($p=3$)}

\addplot [color=cyan, line width=2.0pt]
  table[row sep=crcr]{%
180  1.63\\
300  1.56e-1\\
450  8.10e-3\\
630  2.25e-3\\
840  6.4867e-4\\
1080 1.3969e-4\\
};
\addlegendentry{$p$-refinement ($N_\mathrm{el}=30$)}

\end{axis}
\end{tikzpicture}%
}
    \end{subfigure}%
    \begin{subfigure}[b]{0.5\textwidth}
        \resizebox{\textwidth}{!}{\input{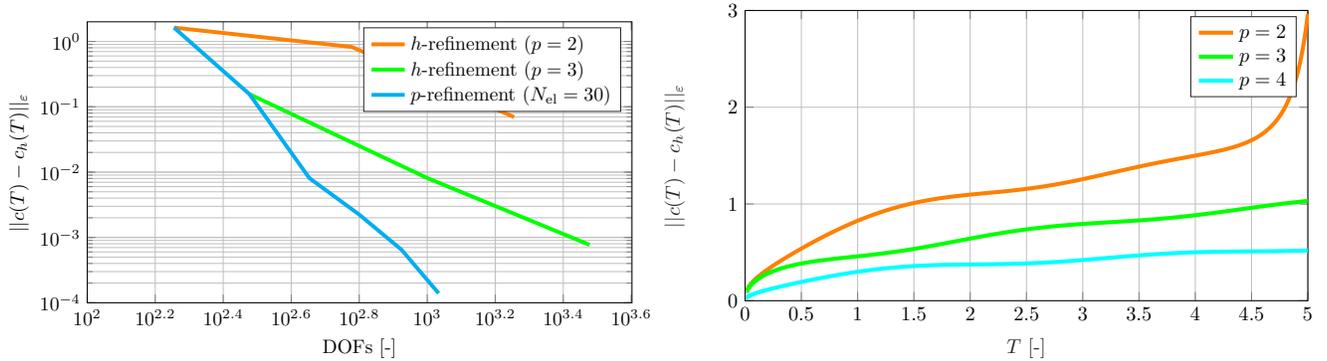}}
    \end{subfigure}%
    \caption{Test case 2: computed errors with respect to DOFs, in different cases of $h$-refinement $(p=2,3)$, and $p$-refinement $(N_\mathrm{el}=30)$ (left), and with respect to the time for $N_\mathrm{el}=30$ and $p=2,3,4$ (right).}
    \label{fig:ErrorsDOF}
\end{figure}

\subsection{Test case 3: Spreading of $\alpha$-synuclein in a 2D brain section}
In this section, we address a numerical simulation of the spreading of the $\alpha$-synuclein on a polygonal agglomerated grid. Starting from structural Magnetic Resonance Images (MRI) of a brain from the OASIS-3 database \cite{OASIS3} we segment the brain by means of Freesurfer \cite{Freesurfer}. After that, we construct a mesh of a slice of the brain along the sagittal plane by means of VMTK \cite{VMTK:antiga}.
\par
The triangular resulting mesh is composed of $41\,859$ triangles, as in Figure~\ref{fig:Meshfibres} (left). However, the generality of the PolyDG method allows us to use mesh elements of any shape, for this reason, we agglomerate the mesh by using ParMETIS \cite{Parmetis} and we obtain a polygonal mesh of 500 elements, as shown in Figure~\ref{fig:Meshfibres} (middle). The solution is computed by using a polynomial order of discretization $p=4$. With this approach, we can on one hand preserve the quality of the geometry description, save computational time (as the mesh is coarse), and exploit the advantage of using high-order approximation. Concerning the time integration we adopt a timestep $\Delta t = 0.01\,\mathrm{years}$.
\par
In order to construct the axonal component of the diffusion tensor $\mathbf{D}$, we derive the diffusion tensor from DTI medical images by using Freesurfer and Nibabel \cite{Nibabel}. By computing the principal eigenvector $\boldsymbol{n}$ of the imaging-derived tensor, we find the directions of the fibres in Figure~\ref{fig:Meshfibres} (right). In this way, we are able to compute the diffusion tensor as in Equation \eqref{eq:difftensor}. Concerning the parameters of the model, we choose the reaction velocity $\alpha = 0.9/\mathrm{year}$. Moreover, we impose an axonal diffusion, which is 10 times faster than the isotropic one: $d_\mathrm{ext} = 8\,\mathrm{mm}^2/\mathrm{year}$ and $d_\mathrm{axn} = 80\,\mathrm{mm}^2/\mathrm{year}$ \cite{schaferInterplayBiochemicalBiomechanical2019}. We fix $f=0$ and we impose homogeneous Neumann boundary conditions on $\partial \Omega$.
\begin{figure}[t]
	\centering
	{\includegraphics[width=0.9\textwidth]{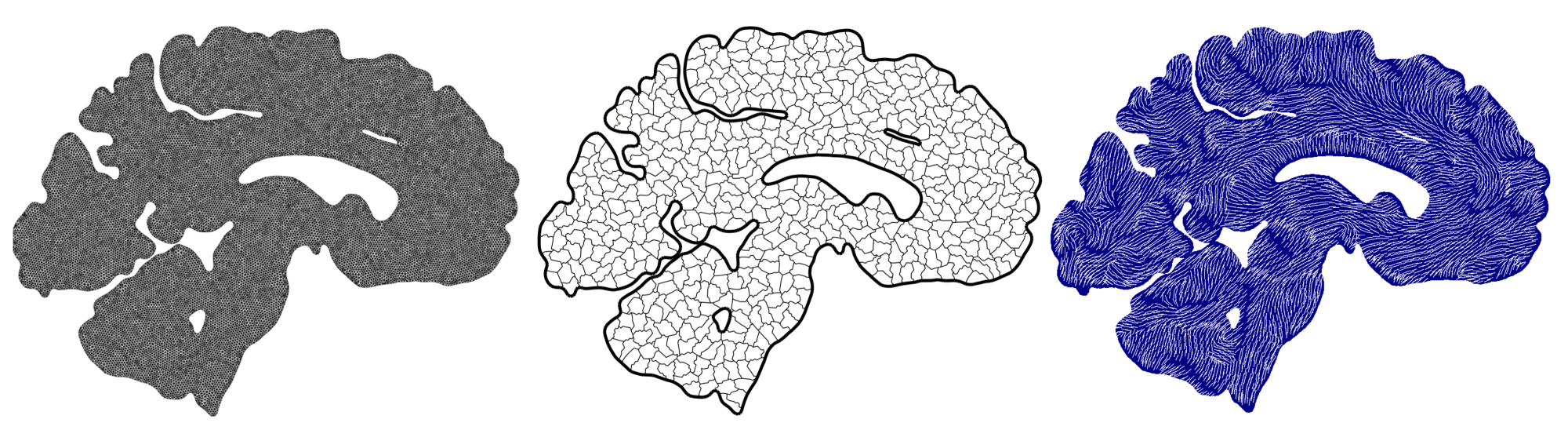}}
	\caption{Fine triangular mesh of a sagittal brain slice (left), agglomerated mesh from the triangular one (centre) and brain reconstructed fibres directions (right).}
	\label{fig:Meshfibres}
\end{figure}
\begin{figure}[t!]
	\centering
	{\includegraphics[width=0.9\textwidth]{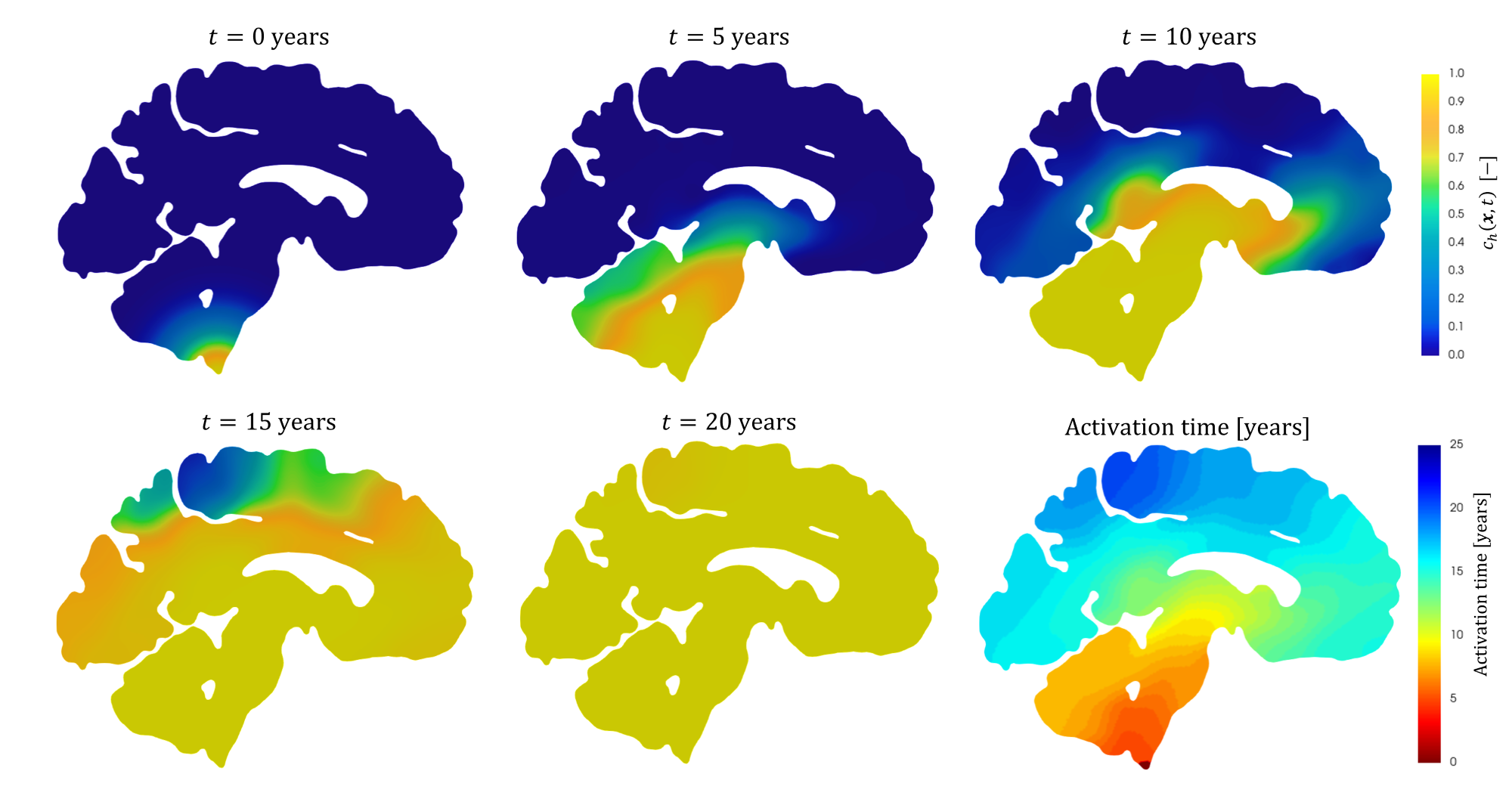}}
	\caption{Patterns of $\alpha$-synuclein concentration at different stages of the pathology and activation time of the pathology (bottom-right).} 
	\label{fig:Solution2DPark}
\end{figure}
\par
To simulate the $\alpha$-synuclein diffusion in Parkinson’s disease we generate an initial condition, with concentration initially located in the dorsal motor nucleus \cite{braakStagingBrainPathology2003}. In Figure \ref{fig:Solution2DPark}, we report both the initial condition (time $t=0$) and the solution at different time instants. We can notice that the diffusion directions are coherent with the medical literature \cite{braakStagingBrainPathology2003, Goedert2015}.
\par
Moreover, we compute the activation time of the pathology as:
\begin{equation}
\label{eq:acttime}
    \hat{t}(\boldsymbol{x},t) = \chi_{\{c_h(\boldsymbol{x},t)>c_\mathrm{crit}\}} (\boldsymbol{x},t) \qquad \boldsymbol{x}\in \Omega \quad t\in[0,T],
\end{equation}
where $\chi$ is the indicator function and $c_\mathrm{crit}$ is the critical value of the pathological protein concentration we fix to be equal to $c_\mathrm{crit}=0.95$. Indeed, a high concentration of misfolded proteins destroys the electric signal transport. This indicator gives us a measure of the time after which the neurons in a specific region will be affected by pathological communication. We report the activation time computed in Figure \ref{fig:Solution2DPark}. We can notice that the time of development of the pathology is of the order of 20 years, coherently with the medical literature \cite{braakStagingBrainPathology2003} and the result is qualitatively similar to other literature results \cite{weickenmeierPhysicsbasedModelExplains2019}.

\subsection{Test case 4: convergence analysis in a 3D case}
\begin{figure}[t!]
    \begin{subfigure}[b]{0.5\textwidth}
          \resizebox{\textwidth}{!}{\definecolor{mycolor2}{rgb}{0.00000,1.00000,1.00000}%
\pgfplotsset{
  log x ticks with fixed point/.style={
      xticklabel={
        \pgfkeys{/pgf/fpu=true}
        \pgfmathparse{exp(\tick)}%
        \pgfmathprintnumber[fixed  zerofill, precision=2]{\pgfmathresult}
        \pgfkeys{/pgf/fpu=false}
      }
  }
}
\begin{tikzpicture}

\begin{axis}[%
width=3.875in,
height=2.36in,
at={(2.6in,1.099in)},
scale only axis,
xmode=log,
xmin=0.108,
xmax=0.866,
xminorticks=true,
xlabel = {$h$ [-]},
ylabel = {$||c(T)-c_h(T)||_\varepsilon$},
ymode=log,
ymin=1e-10,
ymax=0.08,
yminorticks=true,
axis background/.style={fill=white},
title style={font=\bfseries},
xmajorgrids,
xminorgrids,
ymajorgrids,
yminorgrids,
legend style={legend cell align=left, align=left, draw=white!15!black}
]

\addplot [color=red, line width=2.0pt]
  table[row sep=crcr]{%
0.86600  0.06112354102066886\\
0.43300  0.03309810042662088\\
0.21750  0.01710591323524412\\
0.10875  0.00850411814692568\\
};
\addlegendentry{$p=1$}

\addplot [color=orange, line width=2.0pt]
  table[row sep=crcr]{%
0.86600  0.02140436934467019\\
0.43300  0.00523681026752531\\
0.21750  0.00129000089518889\\
0.10875  0.00031641697538347\\
};
\addlegendentry{$p=2$}

\addplot [color=green, line width=2.0pt]
  table[row sep=crcr]{%
0.86600  0.0056330008787462684\\
0.43300  0.0007955610973032714\\
0.21750  0.0001024381561146145\\
0.10875  1.265835939896724e-05\\
};
\addlegendentry{$p=3$}

\addplot [color=mycolor2, line width=2.0pt]
  table[row sep=crcr]{%
0.86600  0.0014241281178261264\\
0.43300  0.0001041160087877188\\
0.21750  6.850647293775494e-06\\
0.10875  4.308483828382257e-07\\
};
\addlegendentry{$p=4$}

\addplot [color=blue, line width=2.0pt]
  table[row sep=crcr]{%
0.86600  0.00032811226451142945\\
0.43300  1.2417694131114222e-05\\
0.21750  4.0288791852359760e-07\\
0.10875  1.2813662788547830e-08\\
};
\addlegendentry{$p=5$}

\addplot [color=purple, line width=2.0pt]
  table[row sep=crcr]{%
0.86600  6.8414514331671680e-05\\
0.43300  1.2743314548563232e-06\\
0.21750  2.1657841948050973e-08\\
0.10875  3.4315259127184004e-10\\
};
\addlegendentry{$p=6$}

\node[right, align=left, text=black, font=\footnotesize]
at (axis cs:0.2005,0.0075) {$1$};

\addplot [color=black, line width=1.5pt]
  table[row sep=crcr]{%
0.200   0.0100\\
0.150   0.0066\\
0.200   0.0066\\
0.200   0.0100\\
};

\node[right, align=left, text=black, font=\footnotesize]
at (axis cs:0.2005,0.0006) {$2$};

\addplot [color=black, line width=1.5pt]
  table[row sep=crcr]{%
0.200   0.000800\\
0.150   0.000452\\
0.200   0.000452\\
0.200   0.000800\\
};

\node[right, align=left, text=black, font=\footnotesize]
at (axis cs:0.2005,4e-5) {$3$};

\addplot [color=black, line width=1.5pt]
  table[row sep=crcr]{%
0.200   0.000060\\
0.150   0.000025\\
0.200   0.000025\\
0.200   0.000060\\
};

\node[right, align=left, text=black, font=\footnotesize]
at (axis cs:0.2005,2e-6) {$4$};

\addplot [color=black, line width=1.5pt]
  table[row sep=crcr]{%
0.200   4.00e-06\\
0.150   1.27e-06\\
0.200   1.27e-06\\
0.200   4.00e-06\\
};

\node[right, align=left, text=black, font=\footnotesize]
at (axis cs:0.2005,8e-8) {$5$};

\addplot [color=black, line width=1.5pt]
  table[row sep=crcr]{%
0.200   2.00e-07\\
0.150   4.80e-08\\
0.200   4.80e-08\\
0.200   2.00e-07\\
};

\node[right, align=left, text=black, font=\footnotesize]
at (axis cs:0.2005,4e-09) {$6$};

\addplot [color=black, line width=1.5pt]
  table[row sep=crcr]{%
0.200   7.00e-09\\
0.150   1.26e-09\\
0.200   1.26e-09\\
0.200   7.00e-09\\
};

\end{axis}
\end{tikzpicture}%
}
    \end{subfigure}%
    \begin{subfigure}[b]{0.5\textwidth}
        \resizebox{\textwidth}{!}{\definecolor{mycolor1}{rgb}{1.00000,1.00000,0.00000}%
\definecolor{mycolor2}{rgb}{0.00000,1.00000,1.00000}%

\begin{tikzpicture}
\begin{axis}[%
width=3.875in,
height=2.36in,
at={(2.6in,1.099in)},
scale only axis,
xmin=1,
xmax=6,
xlabel style={font=\color{white!15!black}},
xlabel={$p$},
ymode=log,
ymin=1e-5,
ymax=0.08,
yminorticks=true,
ylabel style={font=\color{white!15!black}},
ylabel={Error},
axis background/.style={fill=white},
xmajorgrids,
ymajorgrids,
yminorgrids,
legend style={legend cell align=left, align=left, draw=white!15!black}
]

\addplot [color=cyan, line width=2.0pt]
  table[row sep=crcr]{%
1	0.06112354102066886\\
2	0.02140436934467019\\
3	0.00563300087874626\\
4	0.00142412811782612\\
5   0.00032811226451142\\
6	6.8414514331671e-05\\
};
\addlegendentry{$||c(T)-c_h(T)||_\varepsilon$}

\end{axis}
\end{tikzpicture}%
}
    \end{subfigure}%
    \caption{Test case 4: computed errors and convergence rates with respect to $h$ (left) and the polynomial order $p$ (right).}
    \label{fig:errors3D}
\end{figure}
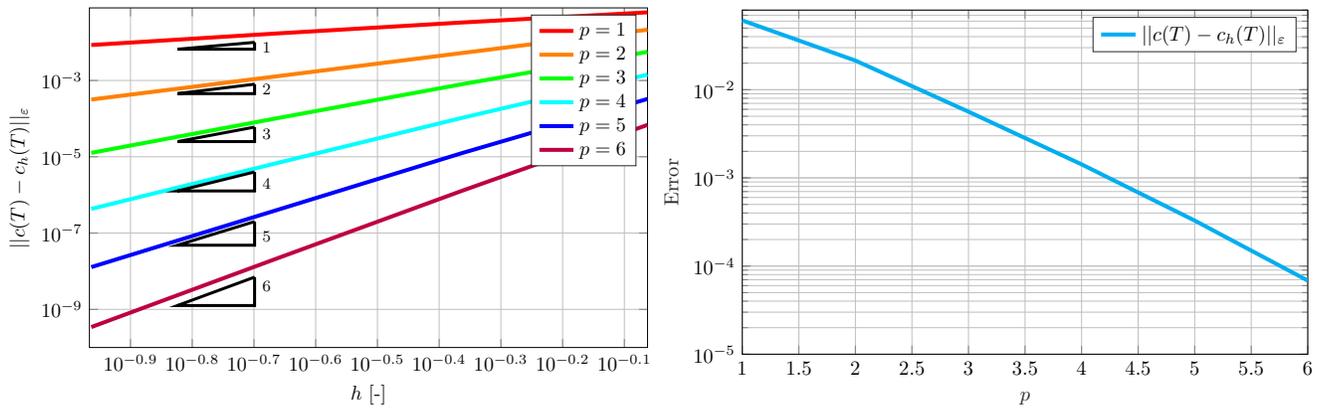
For the numerical tests in this section,  we use the FEniCS finite element software \cite{FenicsCode} (version 2019) to solve the FK equation on tetrahedral meshes. We use a cubic domain $\Omega=(0,1)^3$. Concerning the time discretization, we use a timestep $\Delta t = 10^{-5}$ and a maximum time $T=10^{-3}$. We consider the following manufactured exact solution:
\begin{equation}
    c(x,y,z,t)=\left(\cos(\pi x)\cos(\pi y)\cos(\pi z)\right) e^{-t}.
\end{equation}
In this section we adopt an isotropic diffusion tensor $\mathbf{D}=d_\mathrm{ext}\mathbf{I}$ and we fix the parameters $d_\mathrm{ext}=1$ and $\alpha=0.1$. The forcing term and the Dirichlet boundary condition imposed on $\partial \Omega$ are derived accordingly. The treatment of the nonlinear term in this section is semi-implicit.
\par
In Figure \ref{fig:errors3D}, we report the computed errors in the energy norm defined in Equation \eqref{eq:energynorm} at the final time $T=10^{-3}$. Firstly, we performed the convergence test keeping fixed the polynomial order of the space approximation $p=1,...,6$ and using different mesh refinements $(h=0.866,0.433,0.217,0.108)$. The theoretical rates of convergence are achieved for all the polynomial degrees $p$, coherently to what we proved in Theorem 2.
\par
A convergence analysis with respect to the polynomial order $p$ is also performed with a mesh with $h=0.866$. The results are reported in Figure \ref{fig:errors3D}, where we observe exponential convergence. As we mentioned in the results of the test case of Section 7.1, this case is not covered by our theoretical analysis, nevertheless, numerically we can observe an optimal convergence rate.

\subsection{Test case 5: Spreading of $\alpha$-synuclein in 3D brain}
\begin{figure}
	\centering
	{\includegraphics[width=\textwidth]{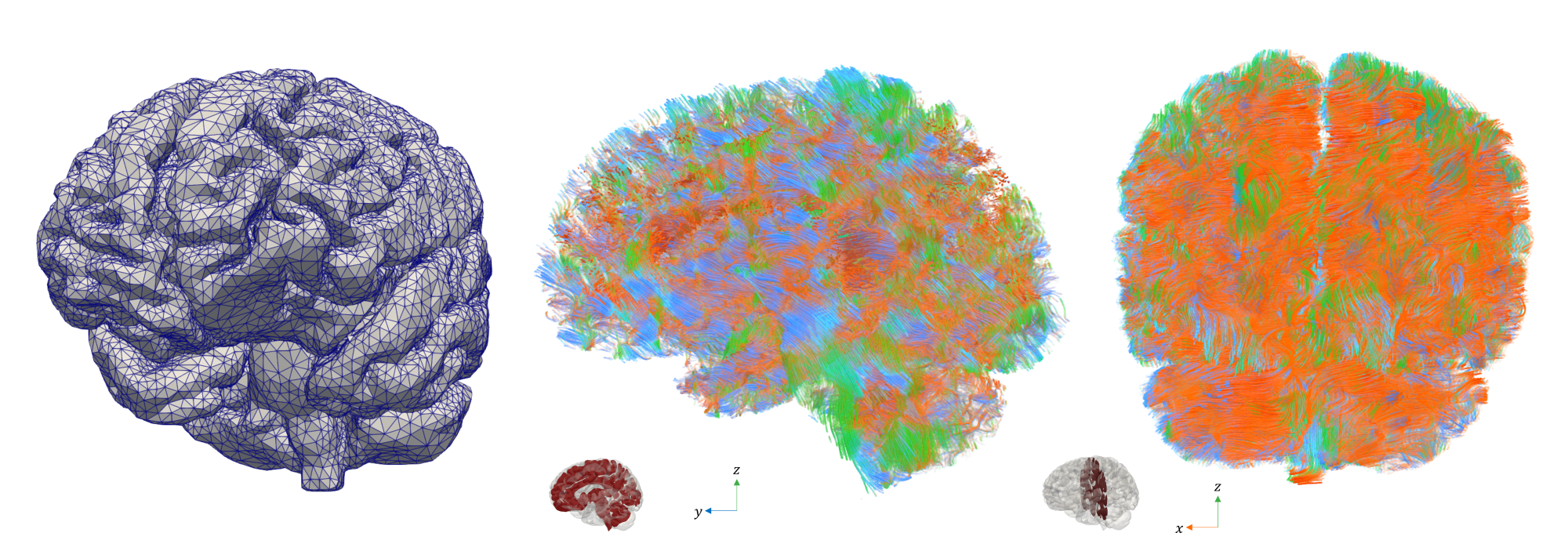}}
	\caption{Brain mesh (left), fibres view from the sagittal plane (centre) and fibres view from the coronal plane (right). In the visualization of the fibres, red indicates directions in the $x$-axis, blue indicates directions in the $y$-axis and green indicates directions in the $z$-axis.}
	\label{fig:Meshfibre3D}
\end{figure}

In this section, we present a numerical simulation of the spreading of the $\alpha$-synuclein on a three-dimensional tetrahedral grid; we use the FEniCS finite element software \cite{FenicsCode} (version 2019). Starting from structural Magnetic Resonance Images (MRI) of a brain from the OASIS-3 database \cite{OASIS3} we segment the brain by means of Freesurfer \cite{Freesurfer}. Finally, the mesh is constructed using the SVMTK library \cite{Mardal:Mesh}. The tetrahedral resulting mesh is composed of 142'658 elements. 
\par
\begin{figure}[t]
	\centering
	{\includegraphics[width=\textwidth]{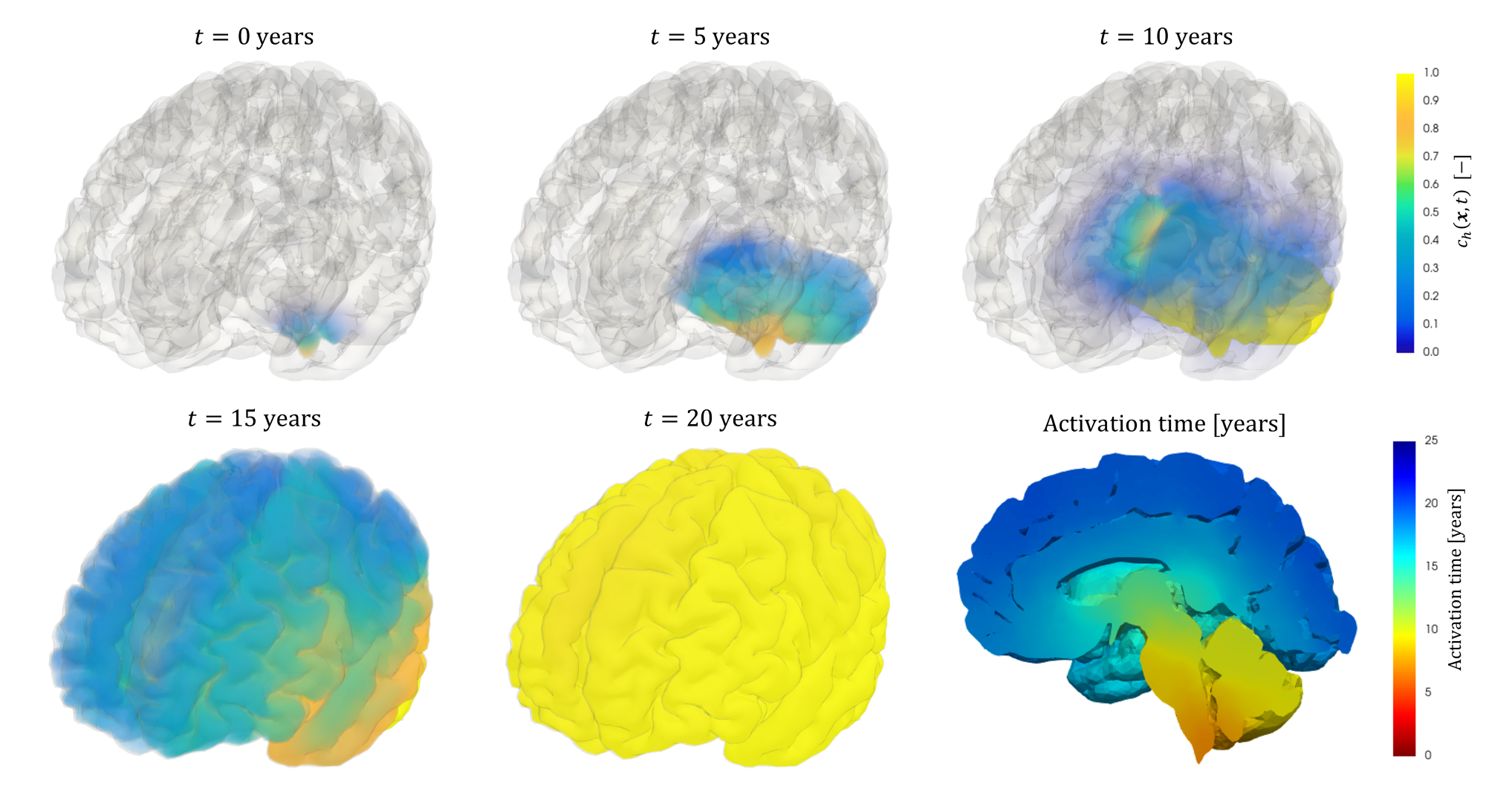}}
	\caption{Patterns of $\alpha$-synuclein concentration at different stages of the pathology with volume rendering and activation time of the pathology on inside the brain geometry (bottom-right).} 
	\label{fig:Solution3DPark}
\end{figure}
The axonal component of the diffusion tensor $\mathbf{D}$ is derived from the diffusion tensor from DTI medical images by using Freesurfer and Nibabel \cite{Nibabel}; the directions of the fibres are reported in Figure~\ref{fig:Meshfibre3D}.  Concerning the parameters of the model, we choose the reaction velocity $\alpha = 0.9/\mathrm{year}$. Moreover, we impose an axonal diffusion, which is 10 times faster than the isotropic one: $d_\mathrm{ext} = 8\,\mathrm{mm}^2/\mathrm{year}$ and $d_\mathrm{axn} = 80\,\mathrm{mm}^2/\mathrm{year}$ \cite{schaferInterplayBiochemicalBiomechanical2019}. Concerning the forcing term we fix $f=0$ and we impose homogeneous Neumann boundary conditions. The solution is computed by means of the PolyDG method with $p=2$. Concerning the time integration we adopt a timestep $\Delta t = 0.01\,\mathrm{years}$.
\par
To simulate the $\alpha$-synuclein diffusion in Parkinson’s disease we generate an initial condition, with concentration initially located in the dorsal motor nucleus \cite{braakStagingBrainPathology2003}, reported in Figure \ref{fig:Solution3DPark}. The simulation gives rise to a propagating front of a misfolded protein concentration. The possibility of increasing the polynomial order is fundamental in this context in order to get a physically consistent solution, without an extremely refined mesh. From  a qualitative point of view, the diffusion directions follow the direction of the reconstructed fibres, as reported in Figure~\ref{fig:Solution3DPark}. Moreover, they are coherent with the medical literature \cite{braakStagingBrainPathology2003, Goedert2015}.
\par
Moreover, we compute the activation time as in Equation \eqref{eq:acttime} and we report it in Figure \ref{fig:Solution3DPark}. We can notice that the time of development of the pathology is of the order of 20 years, coherently with the literature \cite{braakStagingBrainPathology2003, Goedert2015, weickenmeierPhysicsbasedModelExplains2019}.
\par
Finally, we compute the average of the solution $\bar{c}_h(t)$ inside some regions of the brain, which are used in literature to distinguish the 6 Braak's stages \cite{braakStagingBrainPathology2003} in Parkinson's disease. We report the resulting curves over a time interval of 30 years in Figure~\ref{fig:Solutionregions}. We can observe from the region that the initial condition is located inside the dorsal motor nucleus with a mean concentration of $0.2$, then the activation of the regions over the years follows the medical predictions \cite{braakStagingBrainPathology2003}. If we consider as a problematic concentration of $\alpha-$synuclein a value $\bar{c}_h(t)=0.2$, in the mesocortex we reach this value around 13 years. This is an important step, because it can be considered as the beginning of the fourth Braak's stage and then of the symptomatic phase of the disease.
\begin{figure}[t]
	\centering
	{\includegraphics[width=\textwidth]{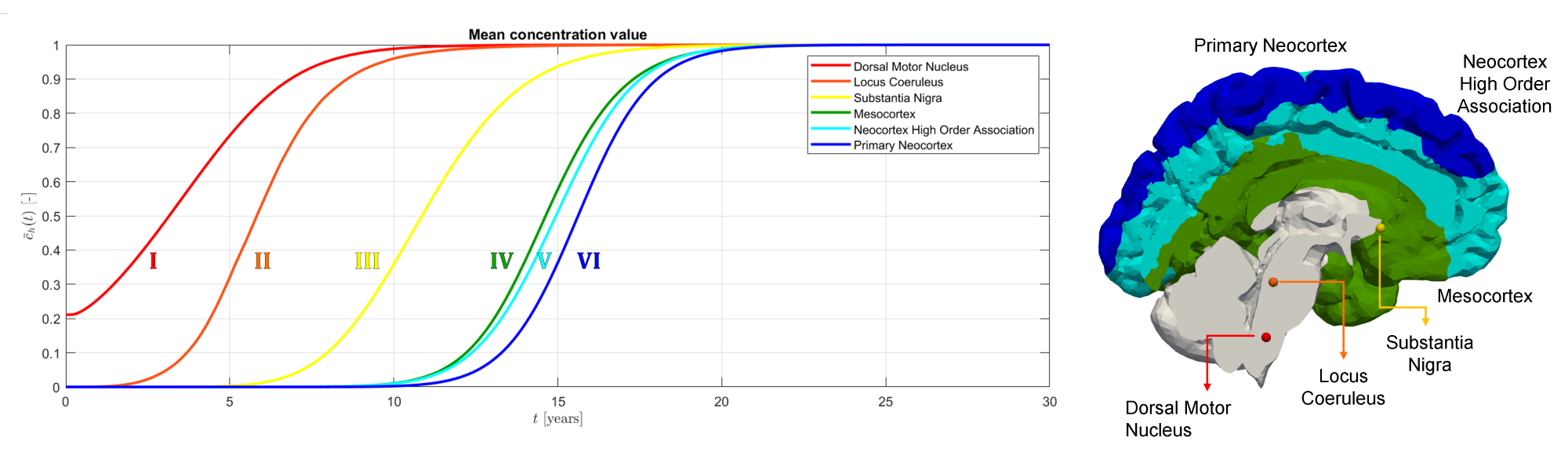}}
	\caption{Mean value of the concentration $\bar{c}_h$ inside some selected regions inside the brain plotted over 30 years (left) and position of brain regions (right).} 
	\label{fig:Solutionregions}
\end{figure}

\section{Conclusions}
\label{sec:conclusion}
In this work, we have proposed a polyhedral discontinuous Galerkin method (PolyDG) for the solution of Fisher-Kolmogorov model applied to the spreading of $\alpha$-synuclein protein in Parkinson's disease. We derived stability and convergence error estimates for arbitrary-order approximation of the semi-discrete formulation. 
\par
The numerical convergence tests were presented both in two and three dimensions. In particular, the convergence tests confirmed the theoretical results of our analysis on polygonal mesh for both implicit and semi-implicit treatments of the nonlinear term. Moreover, we performed a numerical simulation to evaluate the quality of the solution in the case of wavefront propagation in two dimensions. The numerical results confirm the importance of using a high-order method to solve this type of equation maintaining an acceptable computational cost and with a high level of accuracy.
\par
Finally, we present a simulation of $\alpha$-synuclein spreading first on a slice of a real brain in the sagittal plane with a polygonal agglomerated grid and on a 3D brain geometry. We validate the simulations by comparing the activations of some brain regions with the medical literature.
\par
Some future developments of this work can be the construction of DG positivity-preserving schemes on polytopal and polyhedral grids. Moreover, the method can be applied to simulate the spreading of other types of prionic proteins, such as A$\beta$-amyloid and $\tau$. In that context, PET imaging can be used to validate the results. Another interesting future development can be the construction of a space-time DG formulation \cite{antonietti_space-time_2020, antonietti_discontinuous_2023} to achieve higher-order approximations also in time. Finally, it could be interesting to use uncertainty quantification to evaluate the impact of reaction and diffusion parameters on the onset of the disease.

\section*{Acknowledgments}
The brain MRI images were provided by OASIS-3: Longitudinal Multimodal Neuroimaging: Principal Investigators: T. Benzinger, D. Marcus, J. Morris; NIH P30 AG066444, P50 AG00561, P30 NS09857781, P01 AG026276, P01 AG003991, R01 AG043434, UL1 TR000448, R01 EB009352. AV-45 doses were provided by Avid Radiopharmaceuticals, a wholly-owned subsidiary of Eli Lilly.

\section*{Declaration of competing interests}
The authors declare that they have no known competing financial interests or personal relationships that could have appeared to influence the work reported in this article.
\bibliographystyle{ieeetr}
\bibliography{sample.bib}
\end{document}